\documentclass[a4paper,12pt,reqno,table]{amsart}

\usepackage[utf8]{inputenc}
\usepackage[top=25truemm,bottom=25truemm,left=20truemm,right=20truemm]{geometry}

\usepackage{amsthm,amsfonts}
\usepackage{amssymb}
\usepackage{stmaryrd}
\usepackage{mathtools}
\usepackage{mathrsfs}
\usepackage[bbgreekl]{mathbbol}
\usepackage{tikz-cd}
\usepackage{thmtools}
\usepackage{ascmac}
\usepackage{enumitem}

\DeclareSymbolFontAlphabet{\mathbb}{AMSb}
\DeclareSymbolFontAlphabet{\mathbbl}{bbold}

\renewcommand{\epsilon}{\varepsilon}
\renewcommand{\phi}{\varphi}

\setlist[enumerate]{label*=(\roman*)}
\usepackage[%
    backend=biber,%
    style=alphabetic,%
    sorting=nyt,%
    giveninits=true,%
    backref=true,%
]{biblatex}
\definecolor{darkblue}{rgb}{0.2,0,0.6}
\definecolor{darkgreen}{rgb}{0.2,0.5,0.2}
\usepackage[%
    colorlinks=true,%
    linkcolor=darkblue,%
    citecolor=green,%
    urlcolor=darkgreen,%
]{hyperref}
\usepackage[nameinlink]{cleveref}

\crefname{equation}{}{}
\crefname{table}{Table}{Tables}
\crefname{section}{Section}{Sections}

\theoremstyle{definition}
\newtheorem{theorem}{Theorem}[section]
\crefname{theorem}{Theorem}{Theorems}
\newtheorem{definition}[theorem]{Definition}
\crefname{definition}{Definition}{Definitions}
\newtheorem{lemma}[theorem]{Lemma}
\crefname{lemma}{Lemma}{Lemmas}
\newtheorem{corollary}[theorem]{Corollary}
\crefname{corollary}{Corollary}{Corollaries}
\newtheorem{proposition}[theorem]{Proposition}
\crefname{proposition}{Proposition}{Propositions}
\newtheorem{remark}[theorem]{Remark}
\crefname{remark}{Remark}{Remarks}
\newtheorem{example}[theorem]{Example}
\crefname{example}{Example}{Examples}
\newtheorem{notation}[theorem]{Notation}
\crefname{notation}{Notation}{Notations}

\tikzstyle{huge} = [row sep={3.6em}, column sep={3.6em}]
\tikzstyle{large} = [row sep={2.7em}, column sep={2.7em}]
\tikzstyle{normal} = [row sep={1.8em}, column sep={1.8em}]
\tikzstyle{scriptsize} = [row sep={1.35em}, column sep={1.35em}]
\tikzstyle{small} = [row sep={0.9em}, column sep={0.9em}]
\tikzstyle{tiny} = [row sep={0.45em}, column sep={0.45em}]

\usepackage{xpatch}
\makeatletter   
\xpatchcmd{\@tocline}
{\hfil\hbox to\@pnumwidth{\@tocpagenum{#7}}\par}
{\ifnum#1<1\hfill\else\dotfill\fi\hbox to\@pnumwidth{\@tocpagenum{#7}}\par}
{}{}
\def\l@subsection{\@tocline{2}{0pt}{2pc}{6pc}{}}
\makeatother    
\renewcommand{\-}{\mathchar`-}

\newcommand{\A}{\mathscr{A}}
\newcommand{\B}{\mathscr{B}}
\newcommand{\C}{\mathscr{C}}

\newcommand{\E}{\mathscr{E}}

\newcommand{\bA}{\mathbb{A}}
\newcommand{\bB}{\mathbb{B}}

\newcommand{\bF}{\mathbb{F}}
\newcommand{\bG}{\mathbb{G}}
\newcommand{\bI}{\mathbb{I}}

\newcommand{\bN}{\mathbb{N}}
\newcommand{\bS}{\mathbb{S}}
\newcommand{\bT}{\mathbb{T}}
\newcommand{\bX}{\mathbb{X}}
\newcommand{\bZ}{\mathbb{Z}}

\newcommand{\Set}{\mathbf{Set}}
\newcommand{\Cat}{\mathbf{Cat}}
\newcommand{\CAT}{\mathbf{CAT}}
\newcommand{\Mon}{\mathbf{Mon}}
\newcommand{\Grp}{\mathbf{Grp}}
\newcommand{\Ab}{\mathbf{Ab}}
\newcommand{\Ring}{\mathbf{Ring}}
\newcommand{\Pos}{\mathbf{Pos}}
\newcommand{\POS}{\mathbf{POS}}
\newcommand{\Top}{\mathbf{Top}}
\newcommand{\Quiv}{\mathbf{Quiv}}
\newcommand{\End}{\mathbf{End}}
\newcommand{\Cospan}{\mathbf{Cospan}}
\newcommand{\Lat}{\mathbf{Lat}}
\newcommand{\SLat}{\mathbf{SLat}}
\newcommand{\BLat}{\mathbf{Lat_{0,1}}}
\newcommand{\BSLat}{\mathbf{SLat_0}}

\newcommand{\1}{\mathbbl{1}}
\newcommand{\2}{\mathbbl{2}}

\newcommand{\pow}{\mathscr{P}}

\newcommand{\op}{\mathrm{op}}

\newcommand{\ob}{\mathrm{ob}}

\newcommand{\arr}[1]{\overset{#1}{\rightarrow}}
\newcommand{\longarr}[1]{\mathrel{
\tikz\draw[->] (0,0) -- node[above=1.8pt,inner sep=0pt] {\small$#1$} (1,0);
}}
\newcommand{\pto}{\rightharpoonup} 

\newcommand{\const}[1]{\ulcorner #1\urcorner}

\newcommand{\id}{\mathrm{id}}

\newcommand{\Lim}[1]{\mathop{\underset{#1}{\mathrm{Lim}}}}
\newcommand{\Colim}[1]{\mathop{\underset{#1}{\mathrm{Colim}}}}
\newcommand{\Lan}{\mathop{\mathrm{Lan}}}

\newcommand{\incat}[1]{~~\text{in }{#1}}

\newcommand{\fp}[1]{{#1}_\mathrm{fp}}

\newcommand{\card}[1]{\sharp{#1}}

\newcommand{\tup}[1]{\vec{#1}}
\newcommand{\intpn}[2]{\left\llbracket #1 \right\rrbracket_{#2}}
\newcommand{\defined}{\mathord{\downarrow}}
\newcommand{\seq}[1]{\mathrel{
\tikz\draw[|-] (0,0) -- node[above=1.8pt,inner sep=0pt] {\small$#1$} (1,0);
}}
\newcommand{\biseq}[1]{\mathrel{
\tikz\draw[|-|] (0,0) -- node[above=1.8pt,inner sep=0pt] {\small$#1$} (1,0);
}}
\newcommand{\PStr}{\mathrm{PStr}}
\newcommand{\PMod}{\mathrm{PMod}}

\newcommand{\pos}{\mathrm{pos}}

\newcommand{\ar}{\mathrm{ar}}
\newcommand{\type}{\mathrm{type}}
\newcommand{\Alg}{\mathop{\mathit{Alg}}}
\newcommand{\pht}[2]{{\bT_{#1,#2}}} 

\newcommand{\cloprod}{\mathbf{P}}
\newcommand{\clocsub}[1]{\mathbf{S_{#1}}}
\newcommand{\clolret}[1]{\mathbf{H_{#1}}}

\newcommand{\conn}[1]{{#1_{\mathrm{conn}}}}
\newcommand{\Fam}[1]{\mathbf{Fam}(#1)}

\newcommand{\lowset}[1]{\mathop{\downarrow}#1}
\newcommand{\lowlat}[1]{\mathbb{L}(#1)}
\newcommand{\eslowlat}[1]{\mathrm{Es}\mathbb{L}(#1)}
\newcommand{\fglowlat}[1]{\mathrm{Fg}\mathbb{L}(#1)}
\newcommand{\plowlat}[1]{\mathrm{P}\mathbb{L}(#1)}

\newcommand{\atom}[1]{\mathrm{Atom}(#1)}
\newcommand{\opensub}[1]{\mathcal{O}\mathrm{Sub}(#1)}
\newcommand{\sub}[1]{\mathrm{Sub}(#1)}
\newcommand{\aut}[1]{\mathrm{Aut}(#1)}

\newcommand{\fpconn}[1]{{#1}_{\mathrm{fpc}}}

\newcommand{\degree}{\mathop{\mathrm{deg}}}

\newcommand{\multrow}[1]{
\begin{tabular}{c}
    #1
\end{tabular}
}

\title{Filtered colimit elimination from Birkhoff's variety theorem}
\author{Yuto Kawase}
\address{Research Institute for Mathematical Sciences, Kyoto University, Kyoto 606-8502, Japan}
\email{ykawase@kurims.kyoto-u.ac.jp}
\date{\today}
\thanks{The author wishes to express his thanks to his supervisor Masahito Hasegawa for his support and acknowledges many helpful suggestions of Kengo Hirata and Ryuya Hora.}

\begin{document}
\begin{abstract}
    Birkhoff's variety theorem, a fundamental theorem of universal algebra, asserts that a subclass of a given algebra is definable by equations if and only if it satisfies specific closure properties.
    In a generalized version of this theorem, closure under filtered colimits is required.
    However, in some special cases, such as finite-sorted equational theories and ordered algebraic theories, the theorem holds without assuming closure under filtered colimits.
    We call this phenomenon ``filtered colimit elimination,'' and study a sufficient condition for it.
    We show that if a locally finitely presentable category $\A$ satisfies a noetherian-like condition, then filtered colimit elimination holds in the generalized Birkhoff's theorem for algebras \mbox{relative to $\A$.}
\end{abstract}

\maketitle

\tableofcontents

\section{Introduction}

The relative algebraic theories introduced in \cite{kawase2023birkhoffs} are algebraic concepts relative to locally finitely presentable categories, and they are also a syntactic description of finitary monads on locally finitely presentable categories.
Since it is well-known that finitary monads on $\Set$, the category of sets, are equivalent to classical equational theories \cite{linton1969outline}, relative algebraic theory can be regarded as a generalization of equational theory.
That is, classical equational theories are $\Set$-relative algebraic theories in \cite{kawase2023birkhoffs}.

Birkhoff's variety theorem \cite{birkhoff1935structure} is a fundamental result for equational theories and characterizes equational classes via specific closure properties.
A generalization of Birkhoff's variety theorem to relative algebraic theory appears in \cite{kawase2023birkhoffs}.
In such a generalized Birkhoff's theorem, the closure property under filtered colimits is required even though it is not required in the original version of Birkhoff's theorem \cite{birkhoff1935structure}.
In other words, the closure property under filtered colimits can be eliminated from the assumption of the $\Set$-relative case.
We shall call such a phenomenon \emph{filtered colimit elimination}.
Moreover, filtered colimit elimination holds not only in the $\Set$-relative case but also in some other versions.
The first such example is the case of finite-sorted equational theories \cite{adamek2012birkhoffs}, called $\Set^n$-relative algebraic theories in \cite{kawase2023birkhoffs}.
The second example is the case of ordered algebraic theories \cite{bloom1976varieties}, called $\Pos$-relative algebraic theories in \cite{kawase2023birkhoffs}.

In this paper, we study a sufficient condition on a locally finitely presentable category $\A$ for which filtered colimit elimination holds in $\A$-relative algebraic theories.
Our sufficient condition is like a noetherian condition.
A category $\A$ satisfies the \emph{ascending chain condition} (ACC) if it has no strictly ascending sequence (\cref{def:ACC_for_cat}).
We will see that if a locally finitely presentable category $\A$ satisfies ACC, then filtered colimit elimination holds in Birkhoff's theorem for $\A$-relative algebraic theories (\cref{thm:filcolim_elim}), which is our main theorem.
More precisely, that is stated in the following form:
\begin{theorem}[Filtered colimit elimination]\label{thm:intro_main_theorem}
    Let $\rho\colon\bS\to\bT$ be a theory morphism between partial Horn theories.
    Assume that $\bS\-\PMod$ satisfies the ascending chain condition.
    Then, for every replete full subcategory $\E\subseteq\bT\-\PMod$, the following are equivalent:
    \begin{enumerate}
        \item\label{thm:intro_main_theorem-1}
        $\E$ is definable by $\rho$-relative judgments.
        \item\label{thm:intro_main_theorem-2}
        $\E\subseteq\bT\-\PMod$ is closed under products, $\bT$-closed subobjects and \emph{$U^\rho$-local retracts}.
        \item\label{thm:intro_main_theorem-3}
        $\E\subseteq\bT\-\PMod$ is closed under products, $\bT$-closed subobjects, $U^\rho$-retracts and filtered colimits.
    \end{enumerate}
\end{theorem}

Note that the equivalence between \ref{thm:intro_main_theorem-1} and \ref{thm:intro_main_theorem-3} is the generalized Birkhoff's theorem stated in \cite{kawase2023birkhoffs}.
In the condition \ref{thm:intro_main_theorem-2}, closure under filtered colimits is eliminated, however, $U^\rho$-retracts are replaced with \emph{$U^\rho$-local retracts}, which are broader concept than $U^\rho$-retracts.
We emphasize that such replacement is necessary.
In fact, even if the base category $\bS\-\PMod$ satisfies the ascending chain condition, closure under products, $\bT$-closed subobjects, and $U^\rho$-retracts does not imply closure under filtered colimits (\cref{rem:locret_is_essential}).

We now give an overview of the contents of the paper.
In \cref{section2}, we recall some basic notions such as partial Horn theories \cite{palmgren2007partial} and relative algebraic theories \cite{kawase2023birkhoffs}.
We begin in \cref{section3} to recall the notion of \emph{local retractions} (also called \emph{pure quotients} in \cite{adamek2004purequotient}) and to prove some basic properties and a characterization of them.
Local retractions are a kind of epimorphism and will play a crucial role in filtered colimit elimination as described above.
Moreover, we will see that in the hierarchy of epimorphisms, local retractions lie between retractions and regular epimorphisms in a locally finitely presentable category.

In \cref{section4}, we introduce with several examples, the ascending chain condition for categories, which is a sufficient condition for filtered colimit elimination.
This condition is closely related to the strong connectedness of objects.
Two objects $X$ and $Y$ are called \emph{strongly connected} if there are morphisms $X\to Y$ and $Y\to X$.
Now, the ascending chain condition for a category $\A$ is equivalent to saying that the partially ordered class of all strongly connected components in $\A$, denoted by $\sigma(\A)$, satisfies the ordinary ascending chain condition.
More on strong connectedness will be discussed later in \cref{section6}.

In \cref{section5}, we give our main result, filtered colimit elimination described above.
This generalizes some versions of Birkhoff's theorem for the finite-sorted algebras \cite{adamek2012birkhoffs} and the ordered algebras \cite{bloom1976varieties}.
Filtered colimit elimination also provides an important corollary: an HSP-type formalization of the generalized Birkhoff's theorem.
The original version of Birkhoff's theorem not only characterizes equational classes via the closure property under products, subobjects, and quotient but also identifies the closure of a subclass $\E$ under them as $\clolret{}\clocsub{}\cloprod(\E)$, where $\clolret{}$, $\clocsub{}$, $\cloprod$ are the closure operators under quotients, subobjects, products respectively.
By virtue of filtered colimit elimination, we get a similar statement;
the closure of $\E$ under products, $\bT$-closed subobjects and $U^\rho$-local retracts can be simply written as $\clolret{\rho}\clocsub{\bT}\cloprod(\E)$, where $\clolret{\rho}$, $\clocsub{\bT}$, $\cloprod$ are the closure operators under $U^\rho$-local retracts, $\bT$-closed subobjects, products respectively (\cref{cor:HSP-type_formalization}).

At the end of \cref{section5}, we get another surprising result.
We see that under a suitable condition, filtered colimit elimination for $\A$-relative algebraic theories implies the ascending chain condition for the finitely presentable objects in $\A$ (\cref{thm:converse_fil_colim_elim}).
This means that our sufficient condition for filtered colimit elimination is also a nearly necessary condition.

Section \ref{section6} is devoted to some computations of strongly connected components.
It is difficult to determine all strongly connected components in a general category, however, it is somewhat better in a special class of categories, \emph{locally connected categories}.
Locally connected categories are equivalent to the free coproduct-cocompletions of some category \cite{carbonivitale1998}, which particularly include all presheaf categories and categories of topological group actions on a set.
In \cref{section6}, we compute the partially ordered class of strongly connected components in a locally connected category and get a necessary and sufficient condition for a locally connected category to satisfy the ascending chain condition (\cref{cor:loc.conn.cat_ACC}).

For convenience, we summarize the main contributions as follows:
\begin{itemize}
    \item
    We eliminate the closure under filtered colimits from Birkhoff's variety theorem for relative algebraic theories \cite{kawase2023birkhoffs} in some special cases (\cref{thm:filcolim_elim}).
    Consequently, we generalize other Birkhoff-type theorems for finite-sorted algebras \cite{adamek2012birkhoffs} and ordered algebras \cite{bloom1976varieties}.
    \item
    We get an HSP-type formalization of Birkhoff's theorem for relative algebraic theories (\cref{cor:HSP-type_formalization}).
\end{itemize}

\section{Preliminaries}\label{section2}
The content of this section is based on the author's previous preprint \cite{kawase2023birkhoffs}.

\subsection{Partial Horn theories}

We now begin with recalling \emph{partial Horn theories} as presented in \cite{palmgren2007partial}.
For more detail on partial Horn theories, we refer the reader to \cite{palmgren2007partial}.

\begin{definition}
    Let $S$ be a set.
    An \emph{$S$-sorted signature} $\Sigma$ consists of:
    \begin{itemize}
        \item a set $\Sigma_\mathrm{f}$ of function symbols,
        \item a set $\Sigma_\mathrm{r}$ of relation symbols
    \end{itemize}
    such that
    \begin{itemize}
        \item for each $f\in\Sigma_\mathrm{f}$, an arity $f\colon s_1\times\dots\times s_n\to s\,(n\in\bN, s_i,s\in S)$ is given;
        \item for each $R\in\Sigma_\mathrm{r}$, an arity $R\colon s_1\times\dots\times s_n\,(n\in\bN, s_i\in S)$ is given.
    \end{itemize}
\end{definition}

Given the set $S$ of sorts, we fix an $S$-sorted set $\mathrm{Var}=(\mathrm{Var}_s)_{s\in S}$ such that $\mathrm{Var}_s$ is countably infinite for each $s\in S$.
We assume $\mathrm{Var}_s\cap\mathrm{Var}_{s'}=\varnothing$ if $s\neq s'$.
An element $x\in\mathrm{Var}_s$ is called a \emph{variable of sort $s$}.
The notation $x{:}s$ means that $x$ is a variable of sort $s$.

\begin{definition}
    Let $\Sigma$ be an $S$-sorted signature.
    \begin{enumerate}
        \item
        \emph{Terms} (over $\Sigma$) and their types are defined by the following inductive rules:
        \begin{itemize}
            \item
            Given a variable $x$ of sort $s\in S$, $x$ is a term of type $s$;
            \item
            Given a function symbol $f\in\Sigma$ with arity $s_1\times\dots\times s_n\to s$ and terms $\tau_1,\dots,\tau_n$ where $\tau_i$ is of type $s_i$, 
            then $f(\tau_1,\dots,\tau_n)$ is a term of type $s$.
        \end{itemize}
        \item
        \emph{Horn formulas} (over $\Sigma$) are defined by the following inductive rules:
        \begin{itemize}
            \item
            Given a relation symbol $R\in\Sigma$ with arity $s_1\times\dots\times s_n$ and terms $\tau_1,\dots,\tau_n$ where $\tau_i$ is of type $s_i$, 
            then $R(\tau_1,\dots,\tau_n)$ is a Horn formula;
            \item
            Given two terms $\tau$ and $\tau'$ of the same type $s$, 
            then $\tau=\tau'$ is a Horn formula;
            \item
            The truth constant $\top$ is a Horn formula;
            \item
            Given two Horn formulas $\phi$ and $\psi$, $\phi\wedge\psi$ is a Horn formula.
        \end{itemize}
        \item
        A \emph{context} is a finite tuple $\tup{x}$ of distinct variables.
        \item
        A \emph{term-in-context} (over $\Sigma$) is a pair of a context $\tup{x}$ and term $\tau$ (over $\Sigma$), written as $\tup{x}.\tau$, where all variables appearing in $\tau$ occur in $\tup{x}$.
        \item
        A \emph{Horn formula-in-context} (over $\Sigma$) is a pair of a context $\tup{x}$ and Horn formula $\phi$ (over $\Sigma$), written as $\tup{x}.\phi$, where all variables appearing in $\phi$ belong to $\tup{x}$.
        \item
        A \emph{Horn sequent} (over $\Sigma$) is a pair of two Horn formula $\phi$ and $\psi$ (over $\Sigma$) in the same context $\tup{x}$, written as
        \begin{equation*}
            \phi \seq{\tup{x}} \psi.
        \end{equation*}
        \item
        A \emph{partial Horn theory} $\bT$ (over $\Sigma$) is a set of Horn sequents (over $\Sigma$).
    \end{enumerate}
\end{definition}

Note that we do not consider equal sign $=$ to be a relation symbol, and we informally use the abbreviation $\phi\biseq{\tup{x}}\psi$ for ``$(\phi\seq{\tup{x}}\psi)$ and $(\psi\seq{\tup{x}}\phi)$,'' and $\tau\defined$ for $\tau=\tau$.

\begin{definition}
    Let $\bT$ be a partial Horn theory over an $S$-sorted signature $\Sigma$.
    A \emph{partial $\Sigma$-structure} $M$ consists of:
    \begin{itemize}
        \item
        a set $M_s$ for each sort $s\in S$,
        \item
        a partial map $\intpn{f}{M}$ (or $\intpn{\tup{x}.f(\tup{x})}{M}$) $\colon M_{s_1}\times\dots\times M_{s_n}\pto M_s$ for each function symbol $f\colon s_1\times\dots\times s_n\to s$ in $\Sigma$,
        \item
        a subset $\intpn{R}{M}$ (or $\intpn{\tup{x}.R(\tup{x})}{M}$) $\subseteq M_{s_1}\times\dots\times M_{s_n}$ for each relation symbol $R\colon s_1\times\dots\times s_n$ in $\Sigma$.
    \end{itemize}
\end{definition}

We can easily extend the above definitions of $\intpn{\tup{x}.f(\tup{x})}{M}$ and $\intpn{\tup{x}.R(\tup{x})}{M}$ to arbitrary terms-in-context and Horn formulas-in-context.
\begin{itemize}
    \item
    $\intpn{\tup{y}.f(\tau_1,\dots,\tau_n)}{M}(\tup{m})$ is defined if and only if all $\intpn{\tup{y}.\tau_i}{M}(\tup{m})$ are defined and
    \[
    \intpn{f}{M}(\intpn{\tup{y}.\tau_1}{M}(\tup{m}),\dots,\intpn{\tup{y}.\tau_n}{M}(\tup{m}))
    \]
    is also defined, 
    and then
    \[
    \intpn{\tup{y}.f(\tau_1,\dots,\tau_n)}{M}(\tup{m}):=\intpn{f}{M}(\intpn{\tup{y}.\tau_1}{M}(\tup{m}),\dots,\intpn{\tup{y}.\tau_n}{M}(\tup{m}));
    \]
    \item
    $\tup{m}$ belongs to $\intpn{\tup{y}.R(\tau_1,\dots,\tau_n)}{M}$ if and only if all $\intpn{\tup{y}.\tau_i}{M}(\tup{m})$ are defined and
    \[
    (\intpn{\tup{y}.\tau_1}{M}(\tup{m}),\dots,\intpn{\tup{y}.\tau_n}{M}(\tup{m}))
    \]
    belongs to $\intpn{R}{M}$;
    \item
    $\tup{m}$ belongs to $\intpn{\tup{y}.\tau=\tau'}{M}$ if and only if both $\intpn{\tup{y}.\tau}{M}(\tup{m})$ and $\intpn{\tup{y}.\tau'}{M}(\tup{m})$ are defined and equal to each other;
    \item
    $\intpn{\tup{x}.\phi\wedge\psi}{M}:=\intpn{\tup{x}.\phi}{M}\cap\intpn{\tup{x}.\psi}{M}$;
    \item
    $\intpn{\tup{x}.\top}{M}:=\prod_i M_{s_i}$, where $x_i{:}s_i$.
\end{itemize}

\begin{definition}
    We say that a Horn sequent $\phi\seq{\tup{x}}\psi$ over $\Sigma$ is \emph{valid} in a partial $\Sigma$-structure $M$ if $\intpn{\tup{x}.\phi}{M}\subseteq\intpn{\tup{x}.\psi}{M}$.
    A partial $\Sigma$-structure $M$ is called a \emph{partial $\bT$-model} for a partial Horn theory $\bT$ over $\Sigma$ if all sequents in $\bT$ are valid in $M$.
\end{definition}

\begin{definition}
    Let $\Sigma$ be an $S$-sorted signature.
    A \emph{$\Sigma$-homomorphism} $h\colon M\to N$ between partial $\Sigma$-structures consists of:
    \begin{itemize}
        \item a total map $h_s\colon M_s\to N_s$ for each sort $s\in S$
    \end{itemize}
    such that for each function symbol $f\colon s_1\times\dots\times s_n\to s$ in $\Sigma$ and relation symbol $R\colon s_1\times\dots\times s_n$ in $\Sigma$, there exist (necessarily unique) total maps (denoted by dashed arrows) making the following diagrams commute:
    \[
    \begin{tikzcd}[column sep=large, row sep=large]
        M_{s_1}\times\dots\times M_{s_n}\arrow[d,"h_{s_1}\times\dots\times h_{s_n}"'] &[-10pt] \mathrm{Dom}(\intpn{f}{M})\arrow[d,"\exists"',dashed]\arrow[l,hook']\arrow[r,"\intpn{f}{M}"] &[20pt] M_s\arrow[d,"h_s"] \\
        N_{s_1}\times\dots\times N_{s_n} & \mathrm{Dom}(\intpn{f}{N})\arrow[l,hook']\arrow[r,"\intpn{f}{N}"'] & N_s
    \end{tikzcd}
    \]
    \[
    \begin{tikzcd}[column sep=large, row sep=large]
        M_{s_1}\times\dots\times M_{s_n}\arrow[d,"{h_{s_1}\times\dots\times h_{s_n}}"'] &[-10pt] \intpn{R}{M}\arrow[d,"\exists"',dashed]\arrow[l,hook'] \\
        N_{s_1}\times\dots\times N_{s_n} & {\intpn{R}{N}} \arrow[l,hook']
    \end{tikzcd}
    \]
\end{definition}

\begin{notation}
    Let $\bT$ be a partial Horn theory over an $S$-sorted signature $\Sigma$.
    We will denote by $\Sigma\-\PStr$ the category of partial $\Sigma$-structures and $\Sigma$-homomorphisms and by $\bT\-\PMod$ the full subcategory of $\Sigma\-\PStr$ consisting of all partial $\bT$-models.
\end{notation}

\begin{theorem}
    A category $\A$ is locally finitely presentable if and only if there exists a partial Horn theory $\bT$ such that $\A\simeq\bT\-\PMod$.
\end{theorem}
\begin{proof}
    See \cite[Theorem 56, Proposition 59]{palmgren2007partial} and \cite[{}1.46]{adamek1994locally}.
\end{proof}

\begin{example}[posets]\label{eg:pht_for_posets}
    We present the partial Horn theory $\bS_\pos$ for posets.
    Let $S:=\{ *\}$, $\Sigma_\pos:=$\mbox{$\{ \le\colon {*}\times{*} \}$}.
    The partial Horn theory $\bS_\pos$ over $\Sigma_\pos$ consists of:
    \begin{gather*}
        \top\seq{x}x\le x,\quad
        x\le y\wedge y\le x\seq{x,y}x=y,\quad
        x\le y\wedge y\le z\seq{x,y,z}x\le z.
    \end{gather*}
Then, we have $\bS_\pos\-\PMod\cong\Pos$, where $\Pos$ denotes the category of partially ordered sets and monotone maps.
\end{example}

\begin{definition}[\cite{kawase2023birkhoffs}]
    Let $\bT$ be a partial Horn theory over an $S$-sorted signature $\Sigma$.
    A monomorphism $A\hookrightarrow B$ in $\bT\-\PMod$ is called \emph{$\bT$-closed} (or \emph{$\Sigma$-closed}) if the following diagrams form pullback squares for all $f,R\in\Sigma$.
    \begin{equation*}
    \begin{tikzcd}
        \mathrm{Dom}(\intpn{f}{A})\arrow[d,hook']\arrow[r,hook]\arrow[rd,pos=0.1,phantom,"\lrcorner"] &[-10pt] A_{s_1}\times\dots\times A_{s_n}\arrow[d,hook'] \\
        \mathrm{Dom}(\intpn{f}{B})\arrow[r,hook] & B_{s_1}\times\dots\times B_{s_n}
    \end{tikzcd}
    \quad\quad
    \begin{tikzcd}
        \intpn{R}{A}\arrow[d,hook']\arrow[r,hook]\arrow[r,hook]\arrow[rd,pos=0.1,phantom,"\lrcorner"] &[-10pt] A_{s_1}\times\dots\times A_{s_n}\arrow[d,hook'] \\
        {\intpn{R}{B}} \arrow[r,hook] & B_{s_1}\times\dots\times B_{s_n}
    \end{tikzcd}
    \end{equation*}
    Given a $\bT$-closed monomorphism $A\hookrightarrow B$, we call $A$ a $\bT$-closed subobject of $B$.
\end{definition}

\begin{example}
    A monotone map $h\colon X\to Y$ in $\Pos(\cong\bS_\pos\-\PMod)$ is $\bS_\pos$-closed if and only if it is an \emph{embedding}, i.e., 
    $h(x)\le h(y)$ implies $x\le y$.
\end{example}

We will denote by $(S,\Sigma,\bT)$ a partial Horn theory $\bT$ over an $S$-sorted signature $\Sigma$.

\begin{definition}
    Consider two partial Horn theories $(S,\Sigma,\bT)$ and $(S',\Sigma',\bT')$.
    A \emph{theory morphism}
    \[
    \rho\colon (S,\Sigma,\bT)\to(S',\Sigma',\bT')
    \]
    consists of:
    \begin{itemize}
        \item
        a map $S\ni s\mapsto s^\rho\in S'$;
        \item
        an assignment to each function symbol $f\colon s_1\times\dots\times s_n\to s$ in $\Sigma$, a function symbol $f^\rho\colon s_1^\rho\times\dots\times s_n^\rho\to s^\rho$ in $\Sigma'$;
        \item
        an assignment to each relation symbol $R\colon s_1\times\dots\times s_n$ in $\Sigma$, a relation symbol $R^\rho: s_1^\rho\times\dots\times s_n^\rho$ in $\Sigma'$
    \end{itemize}
    such that for every axiom $\phi\seq{\tup{x}}\psi$ in $\bT$, the \emph{$\rho$-translation} $\phi^\rho\seq{\tup{x}^\rho}\psi^\rho$ is a PHL-theorem of $\bT'$.
    Here $\tup{x}^\rho=(x_1^\rho{:}s_1^\rho,\dots,x_n^\rho{:}s_n^\rho)$ with $\tup{x}{:}\tup{s}$.
    The \emph{$\rho$-translation} $\phi^\rho$ and $\psi^\rho$ are constructed by replacing all symbols that $\phi$ and $\psi$ include by $\rho$.
\end{definition}

Our definition of theory morphisms seems strict since it completely depends on syntax.
A looser definition is introduced in \cite[Definition 9.10]{palmgren2007partial}.

A theory morphism $\rho:\bT\to\bT'$ induces the forgetful functor $U^\rho:\bT'\-\PMod\to\bT\-\PMod$.
For $M\in\bT'\-\PMod$, the partial $\bT$-model $U^\rho M$ is given by the following:
\begin{itemize}
    \item
    For each sort $s\in S$, define $(U^\rho M)_s:=M_{s^\rho}$;
    \item
    For each function symbol $f$ in $\Sigma$, define $\intpn{f}{U^\rho M}:=\intpn{\tup{x}^\rho.f^\rho}{M}$;
    \item
    For each relation symbol $R$ in $\Sigma$, define $\intpn{R}{U^\rho M}:=\intpn{\tup{x}^\rho.R^\rho}{M}$.
\end{itemize}

\begin{theorem}
    For every theory morphism $\rho:\bT\to\bT'$, the forgetful functor $U^\rho$ has a left adjoint $F^\rho$.
\end{theorem}
\begin{proof}
    See \cite[Theorem 5.4]{palmgren2007partial}.
\end{proof}

\subsection{Relative algebraic theories}
We now recall the relative algebraic theories \cite{kawase2023birkhoffs}, which are a syntactic description of finitary monads on locally finitely presentable categories.
Throughout this subsection, we fix a partial Horn theory $\bS$ over an $S$-sorted signature $\Sigma$.

\begin{definition}\quad
\begin{enumerate}
    \item
    An \emph{$\bS$-relative signature} $\Omega$ is a set $\Omega$ such that for each element $\omega\in\Omega$, a Horn formula-in-context $\tup{x}.\phi$ over $\Sigma$ and a sort $s\in S$ are given.
    The Horn formula-in-context $\tup{x}.\phi$ is called an \emph{arity} of $\omega$ and written as $\ar(\omega)$.
    The sort $s$ is called a \emph{type} of $\omega$ and written as $\type(\omega)$.
    \item
    Given an $\bS$-relative signature $\Omega$, 
    each $\omega\in\Omega$ can be regarded as a function symbol $\omega\colon s_1\times\dots\times s_n\to s$ if $\type(\omega)=s$ and $\ar(\omega)$ is in the context $x_1{:}s_1,\dots,x_n{:}s_n$.
    Denote by $\Sigma+\Omega$ the $S$-sorted signature obtained by adding to $\Sigma$ all $\omega\in\Omega$ in this way.
    A Horn sequent $\phi\seq{\tup{x}}\psi$ over $\Sigma+\Omega$ is called an \emph{$\bS$-relative judgment} if $\phi$ contains no function symbol derived from $\Omega$.
    \item
    An \emph{$\bS$-relative algebraic theory} is a pair $(\Omega,E)$ of an $\bS$-relative signature $\Omega$ and a set $E$ of $\bS$-relative judgments.
\end{enumerate}
\end{definition}

\begin{definition}
    Let $\Omega$ be an $\bS$-relative signature.
    An ($\bS$-relative) \emph{$\Omega$-algebra} $\bA$ consists of:
    \begin{itemize}
        \item a partial $\bS$-model $A$,
        \item for each $\omega\in\Omega$, a map $\intpn{\omega}{\bA}\colon \intpn{\ar(\omega)}{A}\to A_{\type(\omega)}$.
    \end{itemize}
\end{definition}

An $\Omega$-algebra $\bA$ can be regarded as a partial $(\Sigma+\Omega)$-structure by considering $\intpn{\omega}{\bA}$ as a partial map $A_{s_1}\times\dots\times A_{s_n}\pto A_{\type(\omega)}$, where $\ar(\omega)$ is in the context $x_1{:}s_1,\dots,x_n{:}s_n$.
Conversely, a partial $(\Sigma+\Omega)$-structure satisfying all sequents in $\bS$ and the bisequent $\omega(\tup{x})\defined\biseq{\tup{x}}\ar(\omega)$ for each $\omega\in\Omega$ can be regarded as an $\Omega$-algebra.

\begin{definition}
    Let $\Omega$ be an $\bS$-relative signature.
    We say an $\Omega$-algebra $\bA$ satisfies an $\bS$-relative judgment if $\bA$ satisfies it as a partial $(\Sigma+\Omega)$-structure.
\end{definition}

\begin{notation}\quad
\begin{enumerate}
    \item
    Given an $\bS$-relative signature $\Omega$, we will denote by $\Alg\Omega$ the category of $\Omega$-algebras and $(\Sigma+\Omega)$-homomorphisms.
    \item
    Given an $\bS$-relative algebraic theory $(\Omega,E)$, we will denote by $\Alg(\Omega,E)$ the full subcategory of $\Alg\Omega$ consisting of all algebras satisfying all $\bS$-relative judgments in $E$.
    An $\Omega$-algebra belonging to $\Alg(\Omega,E)$ is called an \emph{$(\Omega,E)$-algebra}.
\end{enumerate}
\end{notation}

\begin{definition}\label{def:pht_for_rat}
    Let $(\Omega,E)$ be an $\bS$-relative algebraic theory.
    We define the partial Horn theory $\pht{\Omega}{E}$ over $\Sigma+\Omega$ associated with $(\Omega,E)$ as follows:
    \begin{equation*}
        \pht{\Omega}{E}:=\bS\cup\{\omega(\tup{x})\defined\biseq{\tup{x}}\ar(\omega)\}_{\omega\in\Omega}\cup E
    \end{equation*}
    Then, we have $\Alg(\Omega,E)\cong\pht{\Omega}{E}\-\PMod$.
    Thus, the category $\Alg(\Omega,E)$ of relative algebras is the category of models of a partial Horn theory.
    In particular, $\Alg(\Omega,E)$ is locally finitely presentable.
\end{definition}

\begin{example}
    Let $\bS_\pos$ be the partial Horn theory for posets as in \cref{eg:pht_for_posets}.
    We now present an example of an $\bS_\pos$-relative algebraic theory.
    A \emph{uniquely difference-ordered semiring} \cite{golan2003semiring} is a tuple $(R,+,\cdot,0,1,\ominus)$ consisting of:
    \begin{itemize}
        \item a poset $R$,
        \item total binary operators $+,\cdot\colon R\times R\to R$,
        \item constants $0,1\in R$,
        \item a partial binary operator $\ominus\colon R\times R\pto R$
    \end{itemize}
    which satisfies the following conditions:
    \begin{itemize}
        \item $(R,+,0)$ is a commutative monoid;
        \item $(R,\cdot,1)$ is a monoid;
        \item $a(b+c)=ab+ac$ and $(a+b)c=ac+bc$ for all $a,b,c\in R$;
        \item $0a=0=a0$ for all $a\in R$;
        \item $b\ominus a$ is defined if and only if $a\le b$;
        \item $a+(b\ominus a)=b$ whenever $a\le b$;
        \item $a\le a+b$ for all $a,b\in R$;
        \item $(a+b)\ominus a=b$ for all $a,b\in R$.
    \end{itemize}
    This definition matches that in \cite{golan2003semiring}.
    The uniquely difference-ordered semirings can be redefined as an $\bS_\pos$-relative algebraic theory.
\end{example}

\section{Local retracts in locally finitely presentable categories}\label{section3}
We recall the notion of \emph{local retractions} in locally finitely presentable categories.
This plays an important role later in \cref{section5}.

\begin{definition}
    A morphism $p\colon X\to Y$ in a category $\A$ is called a \emph{local retraction} if for every finitely presentable object $\Gamma\in\A$ and every morphism $f\colon\Gamma\to Y$, there exists a morphism $g\colon\Gamma\to X$ such that $p\circ g=f$.
\begin{equation*}
\begin{tikzcd}
    & X\arrow[d,"p"]\\
    \Gamma\arrow[ur,"\exists g",dashed]\arrow[r,"f"'] & Y
\end{tikzcd} 
\end{equation*}
Given a local retraction $p\colon X\to Y$, we say that $Y$ is a \emph{local retract} of $X$.
\end{definition}

\begin{remark}
    A local retraction is also called an ($\aleph_0$-)\emph{pure quotient} in \cite{adamek2004purequotient}.
\end{remark}

\begin{example}
    In $\Set$, a local retraction is simply a surjection.
\end{example}

\begin{remark}
    In a locally finitely presentable category, by the small object argument, the class of all local retractions is the right class of a weak factorization system.
    The left class of such a weak factorization system consists of all retracts, in the arrow category, of some coprojection $A\to A+\coprod_{i\in I}B_i$ where all $B_i$ are finitely presentable.
\end{remark}

Here are some elementary properties of local retractions.\vspace{2em}
\begin{lemma}\label{lem:basic_property_locret}
    For every category $\A$, the following hold.
    \begin{enumerate}
        \item
        A composition $p\circ h$ is a local retraction whenever both $p$ and $h$ are local retractions.
        \item
        If a composition $p\circ h$ is a local retraction, then so is $p$.
        \item\label{lem:basic_property_locret-2}
        Every retraction is a local retraction.
        \item\label{lem:basic_property_locret-3}
        If $p\colon X\to Y$ is a local retraction and $Y$ is finitely presentable, then $p$ is a retraction.
        \item\label{lem:basic_property_locret-4}
        Local retractions are stable under pullback, i.e., if a pullback square
        \begin{equation*}
            \begin{tikzcd}
                \cdot\arrow[r]\arrow[d,"p'"']\arrow[rd,pos=0.1,phantom,"\lrcorner"] &[-8pt] \cdot\arrow[d,"p"]\\[-5pt]
                \cdot\arrow[r] & \cdot
            \end{tikzcd}
        \end{equation*}
        is given and $p$ is a local retraction, then $p'$ is also a local retraction.
        \item\label{lem:basic_property_locret-5}
        Let $p=\Colim{I\in\bI}p_I$ be a filtered colimit in the arrow category $\A^\to$ such that $p_I$ is a local retraction in $\A$ for all $I\in\bI$.
        Then $p$ is a local retraction in $\A$.
    \end{enumerate}
\end{lemma}
\begin{proof}
    The proof is immediate.
\end{proof}

\begin{proposition}[{\cite[Proposition 3]{adamek2004purequotient}}]\label{prop:locret_filtered_colim_of_retracts}
    Let $\A$ be a locally finitely presentable category.
    Then for a morphism $p\colon X\to Y$ in $\A$, the following are equivalent:
    \begin{enumerate}
        \item\label{prop:locret_filtered_colim_of_retracts-1}
        $p$ is a local retraction.
        \item\label{prop:locret_filtered_colim_of_retracts-2}
        $p$ is a filtered colimit of retractions, i.e., $p$ can be written as a filtered colimit $p=\Colim{I\in\bI}p_I$ in the arrow category $\A^\to$ such that every $p_I$ is a retract in $\A$.
    \end{enumerate}
\end{proposition}
\begin{proof}
    {[\ref{prop:locret_filtered_colim_of_retracts-2}$\implies$\ref{prop:locret_filtered_colim_of_retracts-1}]}
    This follows from \cref{lem:basic_property_locret}\ref{lem:basic_property_locret-2}\ref{lem:basic_property_locret-5}.
    
    {[\ref{prop:locret_filtered_colim_of_retracts-1}$\implies$\ref{prop:locret_filtered_colim_of_retracts-2}]}
    $Y$ can be written as a filtered colimit $Y=\Colim{I\in\bI}Y_I$ of finitely presentable objects.
    For each $I\in\bI$, take the pullback of $p$ along the $I$-th coprojection
    \begin{equation*}
        \begin{tikzcd}
            X_I\arrow[r]\arrow[d,"p_I"']\arrow[rd,pos=0.1,phantom,"\lrcorner"] & X\arrow[d,"p"]\\
            Y_I\arrow[r] & Y
        \end{tikzcd}\incat{\A}.
    \end{equation*}
    By \cref{lem:basic_property_locret}\ref{lem:basic_property_locret-3}, $p_I$ is a retraction.
    Since filtered colimits are stable under pullback in $\A$, we have $p=\Colim{I\in\bI}p_I$.
\end{proof}

\begin{proposition}[{\cite[Proposition 4]{adamek2004purequotient}}]\label{prop:locret_implies_reg.epi}
    In a locally finitely presentable category, every local retraction is a regular epimorphism.
\end{proposition}
\begin{proof}
    Let $p\colon X\to Y$ be a local retraction in a locally finitely presentable category.
    Take the kernel pair of $p$:
    \begin{equation}\label{eq:kernel_pair_of_p}
        \begin{tikzcd}
            E\arrow[r,shift left,"\pi"]\arrow[r,shift right,"\pi'"'] & X\arrow[r,"p"] & Y.
        \end{tikzcd}
    \end{equation}
    $Y$ can be written as a filtered colimit $Y=\Colim{I\in\bI}Y_I$ of finitely presentable objects.
    Taking the base change of \cref{eq:kernel_pair_of_p} along each coprojection $Y_I\to Y$, we have the following:
    \begin{equation}\label{eq:kernel_pair_of_pI}
        \begin{tikzcd}
            E_I\arrow[r,shift left,"\pi_I"]\arrow[r,shift right,"\pi'_I"'] & X_I\arrow[r,"p_I"] & Y_I.
        \end{tikzcd}
    \end{equation}
    The diagram \cref{eq:kernel_pair_of_pI} is now the kernel pair of $p_I$.
    Since $p_I$ is a retraction, particularly a regular epimorphism, \cref{eq:kernel_pair_of_pI} forms a coequalizer.
    Since the base changes along $p$ and $p\pi(=p\pi')$ preserve the filtered colimit $Y=\Colim{I\in\bI}Y_I$, we have $X=\Colim{I\in\bI}X_I$ and $E=\Colim{I\in\bI}E_I$.
    By commutativity of filtered colimits with coequalizers, \cref{eq:kernel_pair_of_p} forms a coequalizer, hence $p$ is a regular epimorphism.
\end{proof}

\begin{remark}
    In a category that is not locally finitely presentable, \cref{prop:locret_implies_reg.epi} can fail, i.e., a local retraction is not necessarily a regular epimorphism.
    Indeed, in the category $\Top$ of topological spaces and continuous maps, a finitely presentable object is simply a finite discrete space (see \cite[1.2 Examples (10)]{adamek1994locally}\footnote{There is an error in the argument of \cite{adamek1994locally}. A proof that seems to be correct can be found in the answer to a question in mathoverflow: \url{https://mathoverflow.net/q/426127}.}), hence a local retraction in $\Top$ is simply a surjective continuous map.
    However, there is a surjective continuous map which is not a regular epimorphism in $\Top$.
\end{remark}

\begin{example}\label{eg:locret_pos}
    Let $\Pos$ be the category of partially ordered sets and monotone maps.
    Considering the minimal limit ordinal $\omega=\Colim{n<\omega}n\in\Pos$ as a colimit of finite ordinals, we have a canonical morphism $p\colon\coprod_{n<\omega}n\to\omega$ from the coproduct.
    This $p$ is a local retraction in $\Pos$.
    Indeed, since $\omega=\Colim{n<\omega}n$ is a filtered colimit, every morphism $f\colon X\to\omega$ from finitely presentable object $X$ factors through some finite ordinal $n$, i.e., $f$ can be written as $f\colon X\arr{f_n}n\to\omega$ for some $f_n$.
    Then $X\arr{f_n}n\to\coprod_{n<\omega}n$ gives a lift of $f$ along $p$.
\end{example}

\begin{example}\label{eg:locret_end}
    Let $\End$ be the category of sets with an endomorphism, i.e., an object in $\End$ is a pair $(A,a)$ of a set $A$ and an endomorphism $a\colon A\to A$ and a morphism from $(A,a)$ to $(B,b)$ is a map $f\colon A\to B$ such that $b\circ f=f\circ a$.
    Considering $\bN$ and $\bZ$ as objects in $\End$ by the successors $x\mapsto x+1$, the projection $\bZ\times\bN\to\bZ$ is a local retraction in $\End$.
\end{example}

\begin{remark}\quad
\begin{enumerate}
    \item
    \cref{lem:basic_property_locret}\ref{lem:basic_property_locret-2} says: retraction$\implies$local retraction.
    However, the converse does not hold.
    Indeed, the local retractions introduced in \cref{eg:locret_pos,eg:locret_end} are not retractions.
    \item
    \cref{prop:locret_implies_reg.epi} says: in a locally finitely presentable category, local retraction$\implies$regular epimorphism.
    However, the converse does not hold either.
    To show this, consider the following coequalizer in $\Cat$:
    \begin{equation*}
        \begin{tikzcd}
            \1\arrow[r,shift left,"\const{0}"]\arrow[r,shift right,"\const{1}"'] & \2\arrow[r,"q"] & \Sigma\bN.
        \end{tikzcd}
    \end{equation*}
    Here, $\1$ is the terminal, $\2$ is the category with two objects $0,1$ and a unique morphism between them, $\Sigma\bN$ is the category obtained by regarding the additive monoid $\bN$ as a single-object category, and $q$ is the functor choosing an endomorphism $1$ in $\Sigma\bN$.
    Since $q$ is not surjective on morphisms, $\Cat(\2,\2)\arr{q\circ-}\Cat(\2,\Sigma\bN)$ is not surjective, hence $q$ is not a local retraction.
\end{enumerate}
\end{remark}

\section{The ascending chain condition for categories}\label{section4}
We study \emph{strongly connected components} in categories and introduce a noetherian-like condition for a category.
Later, we will show that such a condition is sufficient and nearly necessary to eliminate the closure property under filtered colimits from Birkhoff's theorem.

\begin{definition}\quad
    \begin{enumerate}
        \item
        Objects $X$ and $Y$ in a category are \emph{strongly connected} if there exist morphisms $X\to Y$ and $Y\to X$.
        \item
        Strong connectedness is an equivalence relation on the class of all objects.
        An equivalence class under strong connectedness is called a \emph{strongly connected component}.
    \end{enumerate}
\end{definition}

\begin{remark}
    Every (not necessarily small) poset can be considered as a category by regarding each order as a morphism, which yields a functor $\iota\colon\POS\to\CAT$.
    Here, $\POS$ is the category of large posets, and $\CAT$ is the category of large categories.
    The functor $\iota$ has a left adjoint $\sigma\colon\CAT\to\POS$, which is called the \emph{posetification}.
    The underlying class of the large poset $\sigma(\C)$ is the class of all strongly connected components in $\C$.
    There is an order $[X]\le [Y]$ in $\sigma(\C)$ if and only if there is a morphism $X\to Y$ in $\C$.
\end{remark}

\begin{remark}
    The posetification $\sigma(\C)$ of a category $\C$ can be constructed in a more explicit way:
    A \emph{lower class} in $\C$ is a subclass $L\subseteq\ob\C$ such that if $Y\in L$ and there is a morphism $X\to Y$ in $\C$, then $X\in L$.
    A lower class $L$ in $\C$ is \emph{principal} if $L=\lowset{C}$ for some object $C\in\C$, where 
    \[
    \lowset{C}:=\{ X\in\C \mid \text{there is a morphism } X\to C \text{ in }\C \}.
    \]
    Let $\lowlat{\C}$ denote the large poset of all lower classes in $\C$, whose order is given by inclusion.
    Let $\plowlat{\C}\subseteq\lowlat{\C}$ denote the large subposet consisting of all principal lower classes.
    Then, we get an isomorphism $\sigma(\C)\cong\plowlat{\C}$.
\end{remark}

\begin{remark}\label{rem:posetification_preserves_products}
    The posetification $\sigma\colon\CAT\to\POS$ preserves not only colimits but also products.
    Indeed, whether two objects in a product category are strongly connected is completely determined component-wise.
\end{remark}

\begin{proposition}\label{prop:full-esssurj-srgconn}
    Let $F\colon \A\to\B$ be a functor.
    \begin{enumerate}
        \item\label{prop:full-esssurj-srgconn-1}
        If $F$ is full, the induced map $\sigma(F)\colon \sigma(\A)\to\sigma(\B)$ in $\POS$ is an embedding, i.e.,
        \[
        \sigma(F)([X])\le\sigma(F)([Y]) \incat{\sigma(\B)} \quad\text{ if and only if }\quad [X]\le [Y] \incat{\sigma(\A)},
        \]
        and in particular, $\sigma(F)$ is injective.
        \item\label{prop:full-esssurj-srgconn-2}
        If $F$ is essentially surjective, the induced map $\sigma(F)\colon \sigma(\A)\to\sigma(\B)$ in $\POS$ is surjective.
    \end{enumerate}
    Consequently, we have $\sigma(\A)\cong\sigma(\B)$ whenever $F$ is an equivalence, or full and essentially surjective.
\end{proposition}
\begin{proof}
    The proof is straightforward.
\end{proof}

\begin{definition}\label{def:ACC_for_cat}
    We say a category $\C$ satisfies the \emph{ascending chain condition} (ACC) if the poset $\sigma(\C)$ satisfies the ordinary ascending chain condition, i.e., $\sigma(\C)$ has no infinite strictly ascending sequence.
    Equivalently, a category $\C$ satisfies ACC if for every $\bbomega$-chain $X_0\to X_1\to\cdots$ in $\C$, there exists $N\in\bN$ such that $(X_n)_{n\ge N}$ are strongly connected to each other.
\end{definition}

\begin{proposition}\label{prop:elementary_prop_ACC}\quad
    \begin{enumerate}
        \item\label{prop:elementary_prop_ACC-1}
        If both $\A$ and $\B$ satisfy ACC, so does the product category $\A\times\B$.
        \item\label{prop:elementary_prop_ACC-2}
        If $\A\times\B$ satisfies ACC and $\B$ is non-empty, then $\A$ satisfies ACC.
        \item\label{prop:elementary_prop_ACC-3}
        If a functor $F\colon\A\to\B$ is full and $\B$ satisfies ACC, then $\A$ also satisfies ACC.
    \end{enumerate}
\end{proposition}
\begin{proof}
    \ref{prop:elementary_prop_ACC-1} and \ref{prop:elementary_prop_ACC-2} follow from the fact that the posetification $\sigma$ preserves products (\cref{rem:posetification_preserves_products}).
    \ref{prop:elementary_prop_ACC-3} follows from \cref{prop:full-esssurj-srgconn}\ref{prop:full-esssurj-srgconn-1}.
\end{proof}

\begin{example}\label{eg:srg_conn}\quad
    \begin{enumerate}
        \item
        The categories of pointed sets, monoids, groups, abelian groups, bounded semilattices, etc., are strongly connected, that is, they have a unique strongly connected component.
        This is because every object in these categories has a morphism from the terminal object.
        In particular, these categories satisfy ACC, since every finite poset satisfies ACC.
        \item
        $\Set$ has only two strongly connected components, the empty set and the non-empty sets.
        More precisely, $\sigma(\Set)\cong \2=\{0<1\}$.
        The same argument holds for the categories of posets, small categories, lattices, semilattices, etc.
        In particular, these categories satisfy ACC.
        \item
        The arrow category $\Set^\to$ has only three strongly connected components:
        \[
        \sigma(\Set^\to)=\{~~ [\varnothing\to\varnothing]~<~[\varnothing\to 1]~<~[1\to 1] ~~\}.
        \]
        \item
        Let $S$ be a set.
        Since the posetification preserves products, we get
        \[
        \sigma(\Set^S)\cong\sigma(\Set)^S\cong\2^S\cong\pow(S),
        \]
        where $\pow(S)$ is the power set of $S$.
        Thus, $\Set^S$ satisfies ACC if and only if $S$ is finite.
        \item
        Let $S$ be a non-empty set.
        The coslice category $S/\Set$ can be considered as the category of sets with $S$-indexed constants by regarding each object as a set $X$ with constants $\intpn{c_s}{X}\in X~~(s\in S)$.
        Let $X,Y$ be objects in $S/\Set$.
        If $\intpn{c_s}{X}=\intpn{c_t}{X}$ and $\intpn{c_s}{Y}\neq\intpn{c_t}{Y}$ hold for some $s,t\in S$, then there is no morphism $X\to Y$ in $S/\Set$.
        The converse is also true.
        Indeed, if $X$ and $Y$ merge their constants in the same way, there are morphisms $X\to Y$ and $Y\to X$.
        Therefore, there is a bijective correspondence between strongly connected components of $S/\Set$ and equivalence relations on $S$, hence the cardinal $\card{\sigma(n/\Set)}$ equals the $n$-th Bell number.
        Moreover, we can see that $S/\Set$ satisfies ACC if and only if $S$ is finite.
        \item
        The category of bounded lattices does not satisfy ACC.
        To see this, consider the following lattices:
        \begin{equation*}
            M_n:=\left(
            \begin{tikzcd}[tiny]
                & & 1\arrow[lld,no head]\arrow[ld,no head]\arrow[d,no head]\arrow[rrd,no head] & & \\[10pt]
                a_0 & a_1 & a_2 & \cdots & a_{n-1} \\[10pt]
                & & 0\arrow[llu,no head]\arrow[lu,no head]\arrow[u,no head]\arrow[rru,no head] & &
            \end{tikzcd}
            \right)\quad\quad (n\ge 2).
        \end{equation*}
        Then, we get an $\bbomega$-chain $M_2\to M_3\to\cdots$ of bounded lattices, and there is no bounded lattice morphism $M_n\to M_m$ whenever $n>m$.
        \item
        The category of rings, denoted by $\Ring$, has infinitely many strongly connected components.
        In fact, the rings $\bZ/n\bZ~~(n\in\bN)$ are not strongly connected to each other.
        Moreover, $\Ring$ does not satisfy ACC.
        This is because there is a non-trivial $\bbomega$-chain of finite fields:
        \begin{equation*}
            \bF_p\hookrightarrow\bF_{p^2}\hookrightarrow\bF_{p^4}\hookrightarrow
            \cdots\hookrightarrow\bF_{p^{2^n}}\hookrightarrow\cdots,
        \end{equation*}
        where $p$ is a prime number and $\bF_{p^{2^n}}$ is the finite field of order $p^{2^n}$.
        \item
        The category of semigroups does not satisfy ACC.
        To see this, let $S_n$ denote the semigroup generated by $n+1$ elements $x_0,x_1,\dots,x_n$ with the equations $x_i=x_i x_j~(\forall i>j)$.
        Then, we get an $\bbomega$-chain $S_0\to S_1\to S_2\to\cdots$ of semigroups.
        Note that each element of $S_n$ can be uniquely written as $x_0^{a_0}x_1^{a_1}\cdots x_n^{a_n}$, where $a_i\in\bN$ and at least one of them is non-zero.
        We define the \emph{degree} of an element $\alpha=x_0^{a_0}x_1^{a_1}\cdots x_n^{a_n}\in S_n$, denoted by $\degree\alpha$, as the largest number $i$ such that $a_i\neq 0$.
        
        We now claim that $\alpha=\alpha\beta$ in $S_n$ implies $\degree\alpha >\degree\beta$.
        To show this, take elements $\alpha=x_0^{a_0}x_1^{a_1}\cdots x_N^{a_N}$ $(a_N\neq 0)$ and $\beta=x_0^{b_0}x_1^{b_1}\cdots x_n^{b_n}$ of $S_n$ such that $\alpha=\alpha\beta$.
        Then, it follows that
        \[
            \alpha=\alpha\beta=x_0^{a_0}\cdots x_{N-1}^{a_{N-1}}x_N^{a_N+b_N}x_{N+1}^{b_{N+1}}\cdots x_n^{b_n}.
        \]
        Comparing the exponents, we have $b_N=b_{N+1}=\cdots =b_n=0$, hence $\degree\beta <N=\degree\alpha$.

        If there exists a morphism $f\colon S_n\to S_m$, by the above argument, we have
        \[
            m\ge\degree f(x_n)>\degree f(x_{n-1})>\cdots >\degree f(x_0)\ge 0.
        \]
        This implies $n\le m$.
        Thus, the $\bbomega$-chain $S_0\to S_1\to S_2\to\cdots$ has no morphism in the converse direction.
        \item
        The category $\End$ of sets with an endomorphism does not satisfy ACC.
        To show this, let $C_n:=(\bZ/n\bZ, +1)\in\End$ denote the $n$-cycle, and let $A_n:=\coprod_{1\le k\le n}C_{p_k}$, where $p_k$ denotes the $k$-th prime number.
        Then, we get an $\bbomega$-chain $A_1\to A_2\to\cdots$ of coprojections, and there is no morphism $A_n\to A_m$ whenever $n>m$.
        \item
        The category $\Quiv$ of quivers (or directed graphs) does not satisfy ACC.
        This follows from \cref{prop:elementary_prop_ACC}\ref{prop:elementary_prop_ACC-3} and the fact that there is a fully faithful functor $\End\to\Quiv$ defined by:
        \[
        \End\ni
        \begin{tikzcd}
            A\arrow[loop right,"a"]
        \end{tikzcd}
        \quad\mapsto\quad
        \begin{tikzcd}
            A\arrow[r,bend left=20,"a"]\arrow[r,bend right=20,"\id"'] & A
        \end{tikzcd}
        \in\Quiv.
        \]
        \item
        Let $\bbomega$ denote the category $\{ 0\to 1\to 2\to\cdots \}$.
        Then, the presheaf category $\Set^{\bbomega^\op}$ does not satisfy ACC.
        To see this, define $L_i\in\Set^{\bbomega^\op}~(i\in\bN)$ by
        \begin{equation*}
            L_i(n):=
            \begin{cases*}
                1 & if $n<i$, \\
                \varnothing & if $n\ge i$.
            \end{cases*}
        \end{equation*}
        Then we get a sequence $L_0\to L_1\to\cdots$ in $\Set^{\bbomega^\op}$, and there is no natural transformation $L_n\to L_m$ whenever $n>m$.
        \item
        On the other hand, $\Set^\bbomega$ satisfies ACC:
        For each ordinal $\alpha\le\omega$, define $U_\alpha\in\Set^\bbomega$ by
        \begin{equation*}
            U_\alpha(n):=
            \begin{cases*}
                \varnothing & if $n<\alpha$, \\
                1 & if $n\ge\alpha$.
            \end{cases*}
        \end{equation*}
        Then, the strongly connected components of $\Set^\bbomega$ are displayed as follows:
        \begin{equation*}
            \sigma(\Set^\bbomega)=\{~~[U_\omega]~<\cdots<~[U_2]~<~[U_1]~<~[U_0]~~\}~~\cong(\omega+1)^\op.
        \end{equation*}
        It follows that the poset $(\omega+1)^\op$ satisfies ACC.
        \item\label{eg:srg_conn-cospan}
        Let $\Cospan$ be the category of cospans in $\Set$, i.e., 
        an object in $\Cospan$ is a functor from the category $\{\cdot\rightarrow\cdot\leftarrow\cdot\}$ to $\Set$, and a morphism is a natural transformation.
        Then, the category $\Cospan$ has only six strongly connected components:
        \begin{equation*}
            S_0\colon
            \begin{tikzcd}[tiny]
                & \varnothing & \\
                \varnothing\arrow[ru] & & \varnothing,\arrow[lu]
            \end{tikzcd}\quad
            S_1\colon
            \begin{tikzcd}[tiny]
                & 1 & \\
                \varnothing\arrow[ru] & & \varnothing,\arrow[lu]
            \end{tikzcd}\quad
            S_2\colon
            \begin{tikzcd}[tiny]
                & 1 & \\
                1\arrow[ru] & & \varnothing,\arrow[lu]
            \end{tikzcd}
        \end{equation*}
        \begin{equation*}
            S_3\colon
            \begin{tikzcd}[tiny]
                & 1 & \\
                \varnothing\arrow[ru] & & 1,\arrow[lu]
            \end{tikzcd}\quad
            S_4\colon
            \begin{tikzcd}[tiny]
                & 2 & \\
                1\arrow[ru,"\const{0}"] & & 1,\arrow[lu,"\const{1}"']
            \end{tikzcd}\quad
            S_5\colon
            \begin{tikzcd}[tiny]
                & 1 & \\
                1\arrow[ru] & & 1.\arrow[lu]
            \end{tikzcd}
        \end{equation*}
        \item\label{eg:srg_conn-doblcospan}
        Let ${\bowtie}$ denote the following category:
        \begin{equation*}
            {\bowtie}:=\left(
            \begin{tikzcd}
                \cdot & \cdot \\
                \cdot\arrow[u]\arrow[ur] & \cdot\arrow[u]\arrow[ul,crossing over]
            \end{tikzcd}
            \right)
        \end{equation*}
        Then, the presheaf category $\Set^{\bowtie}$ does not satisfy ACC and, in particular, has infinitely many strongly connected components even though $\sigma(\Cospan)$ is finite.
        This is because there is a fully faithful functor $\End\to\Set^{\bowtie}$:
        \[
        \End\ni (A,a)\quad\mapsto\quad
        \begin{tikzcd}
            A & A \\
            A\arrow[u,equal]\arrow[ur,equal] & A\arrow[u,equal]\arrow[ul,crossing over,"a"',pos=0.8]
        \end{tikzcd}
        \in\Set^{\bowtie}.
        \]
        \item\label{eg:srg_conn-X}
        Let $\bX$ denote the following category:
        \begin{equation*}
            \bX:=\left(
            \begin{tikzcd}[tiny]
                \cdot & & \cdot \\
                & \cdot\arrow[ul]\arrow[ur] & \\
                \cdot\arrow[ur] & & \cdot\arrow[ul]
            \end{tikzcd}
            \right)
        \end{equation*}
        We now get a fully faithful functor $\phi\colon{\bowtie}\to\bX$ by ignoring the crossing point in $\bX$.
        Then, the global left Kan extension $\Lan_\phi\colon \Set^{\bowtie}\to\Set^\bX$ along $\phi$ is fully faithful, hence the category $\Set^\bX$ also does not satisfy ACC.
    \end{enumerate}
\end{example}

\cref{tab:ACC_for_LFP} at the end of this paper summarizes the above examples and includes more.

\section{Filtered colimit elimination}\label{section5}
Our goal is to prove \cref{thm:filcolim_elim}, \emph{filtered colimit elimination}.
This is our main theorem and generalizes some Birkhoff-type theorems for finite-sorted algebras \cite{adamek2012birkhoffs} and ordered algebras \cite{bloom1976varieties}.

\begin{lemma}\label{lem:diagram_strongly_connected}
    Let $\C$ be a category which satisfies ACC.
    Let $D\colon\bI\to\C$ be a non-empty diagram.
    Then, there exists an object $I_0\in\bI$ such that for every morphism $I_0\to I$ in $\bI$, $DI_0$ and $DI$ are strongly connected in $\C$.
\end{lemma}
\begin{proof}
    Assume that for every object $I\in\bI$, there is a morphism $I\to J$ in $\bI$ such that $DI$ and $DJ$ are not strongly connected.
    Since $\bI$ is non-empty, by assumption, we can take an $\bbomega$-chain $I_0\to I_1\to I_2\to\cdots$ in $\bI$ such that for any $n$, $DI_n$ and $DI_{n+1}$ are not strongly connected.
    Then, we have a strictly increasing sequence $\sigma(DI_0)<\sigma(DI_1)<\sigma(DI_2)<\cdots$ in $\sigma(\C)$.
    This contradicts our assumption that $\sigma(\C)$ satisfies ACC.
\end{proof}

\begin{definition}[\cite{kawase2023birkhoffs}]
    Let $\rho\colon (S,\Sigma,\bS)\to (S',\Sigma',\bT)$ be a theory morphism between partial Horn theories.
    A \emph{$\rho$-relative judgment} is a Horn sequent $\phi^\rho\seq{\tup{x}^\rho}\psi$, where $\tup{x}.\phi$ is a Horn formula-in-context over $\Sigma$ and $\tup{x}^\rho.\psi$ is a Horn formula-in-context over $\Sigma'$.
\end{definition}

\begin{definition}
    Let $U\colon\A\to\C$ be a functor.
    A morphism $p$ in $\A$ is called a \emph{$U$-local retraction} if $Up$ is a local retraction in $\C$.
\end{definition}

We can now formulate our main result.

\begin{theorem}[Filtered colimit elimination]\label{thm:filcolim_elim}
    Let $\rho\colon\bS\to\bT$ be a theory morphism between partial Horn theories.
    Assume that $\bS\-\PMod$ satisfies the ascending chain condition.
    Then, for every replete full subcategory $\E\subseteq\bT\-\PMod$, the following are equivalent:
    \begin{enumerate}
        \item\label{thm:filcolim_elim-1}
        $\E$ is definable by $\rho$-relative judgments, i.e., there exists a set $\bT'$ of $\rho$-relative judgments satisfying $\E=(\bT+\bT')\-\PMod$.
        \item\label{thm:filcolim_elim-2}
        $\E\subseteq\bT\-\PMod$ is closed under products, $\bT$-closed subobjects and $U^\rho$-local retracts.
        \item\label{thm:filcolim_elim-3}
        $\E\subseteq\bT\-\PMod$ is closed under products, $\bT$-closed subobjects, $U^\rho$-retracts and filtered colimits.
    \end{enumerate}
\end{theorem}
\begin{proof}
    {[\ref{thm:filcolim_elim-3}$\implies$\ref{thm:filcolim_elim-1}]}
    This follows from \cite[Theorem 3.5]{kawase2023birkhoffs}.
    
    {[\ref{thm:filcolim_elim-1}$\implies$\ref{thm:filcolim_elim-2}]}
    Consider the following adjunction:
    \begin{equation*}
    \begin{tikzcd}[column sep=large, row sep=large]
        \bS\-\PMod\arrow[r,"F^\rho",shift left=7pt]\arrow[r,"\perp"pos=0.5,phantom] &[10pt]\bT\-\PMod\arrow[l,"U^\rho",shift left=7pt]
    \end{tikzcd}
    \end{equation*}
    By the same argument as \cite[Theorem 3.5]{kawase2023birkhoffs}, $(\bT+\bT')\-\PMod\subseteq\bT\-\PMod$ can be presented as an orthogonality class with respect to a family of epimorphisms
    \begin{equation*}
        \Lambda=\{ F^\rho(\Gamma_i)\longarr{e_i}\Delta_i \}_{i\in I}
    \end{equation*}
    such that all $\Gamma_i$ are finitely presentable in $\bS\-\PMod$.
    By \cite[Theorem 3.5]{kawase2023birkhoffs} again, we already know that $(\bT+\bT')\-\PMod\subseteq\bT\-\PMod$ is closed under products and $\bT$-closed subobjects.
    Thus, we only need to show that $(\bT+\bT')\-\PMod\subseteq\bT\-\PMod$ is closed under $U^\rho$-local retracts.
    Let $p\colon A\to B$ be a $U^\rho$-local retraction with $A\in (\bT+\bT')\-\PMod$.
    Given a morphism $f\colon F^\rho(\Gamma_i)\to B$, since $p\colon A\to B$ is a $U^\rho$-local retraction, we have a morphism $g\colon F^\rho(\Gamma_i)\to A$ such that $pg=f$.
    Since $A$ is orthogonal to $e_i$, there is a unique morphism $h\colon\Delta_i\to A$ such that $he_i=g$.
    Then $f=pg=phe_i$ holds, hence $B$ is orthogonal to $e_i$.
    \begin{equation*}
        \begin{tikzcd}[column sep=huge, row sep=huge]
            F^\rho(\Gamma_i)\arrow[d,"e_i"']\arrow[rd,"f"',pos=0.7]\arrow[r,"g"] & A\arrow[d,"p"] \\
            \Delta_i\arrow[ru,"\exists !h",pos=0.7,crossing over,dashed] & B
        \end{tikzcd}\quad\quad\incat{\bT\-\PMod}
    \end{equation*}
    
    {[\ref{thm:filcolim_elim-2}$\implies$\ref{thm:filcolim_elim-3}]}
    It suffices to show that $\E\subseteq\bT\-\PMod$ is closed under filtered colimits.
    Let $A=\Colim{I\in\bI}A_I$ be a filtered colimit in $\bT\-\PMod$ with $A_I\in\E$.
    By \cite[1.5 Theorem]{adamek1994locally}, without loss of generality, we can assume that $\bI$ is a directed poset.
    
    By \cref{lem:diagram_strongly_connected}, there exists an object $I_0\in\bI$ such that $U^\rho (A_{I_0})$ and $U^\rho (A_I)$ are strongly connected in $\bS\-\PMod$ whenever $I_0\le I$ in $\bI$.
    Since $\bI$ is directed, the coslice
    \[
    {\uparrow}I_0:=\{I\mid I_0\le I\}\subseteq\bI
    \]
    is also directed and its inclusion is final.
    Thus, there is no loss of generality in assuming that all $U^\rho(A_I)$ are strongly connected to each other.

    We now define a subposet $\bI_J$ of $\bI$ for each $J\in\bI$.
    The subposet $\bI_J$ contains all objects in $\bI$.
    The ordering of $\bI_J$ is defined by the following:
    \begin{equation*}
        \bI_J(I,I'):=
        \begin{cases*}
            \bI(I,I') & if $J\le I$ in $\bI$, \\
            \{ \id_I \} & if $I=I'$ and $J\not\le I$ in $\bI$, \\
            \varnothing & if $I\neq I'$ and $J\not\le I$ in $\bI$.
        \end{cases*}
    \end{equation*}
    Take the limit $B_J:=\Lim{I\in\bI_J}A_I$ and let $\pi_J\colon B_J\to A_J$ be the $J$-th projection.

    We now show that $\pi_J\colon B_J\to A_J$ is a $U^\rho$-retraction.
    For each $I\in\bI$ with $J\le I$ in $\bI$, we have a morphism $s_I\colon U^\rho(A_J)\longarr{U^\rho A_!}U^\rho(A_I)$, where $!\colon J\to I$ is a unique morphism in $\bI$.
    For each $I\in\bI$ with $J\not\le I$ in $\bI$, since all $U^\rho(A_I)$ are strongly connected to each other, we can choose a morphism $s_I\colon U^\rho(A_J)\to U^\rho(A_I)$.
    Then, since $U^\rho$ preserves limits, we have a morphism $s=(s_I)_I\colon U^\rho(A_J)\to \Lim{I\in\bI_J}U^\rho(A_I)=U^\rho(B_J)$, which is a section of $U^\rho(\pi_J)$.

    Note that the canonical morphism $m_J\colon B_J=\Lim{I\in\bI_J}A_I\to\prod_{I}A_I$ is a strong monomorphism and particularly $\bT$-closed.
    By \cite[Theorem 4.14]{kawase2023birkhoffs}, the class of all $\bT$-closed monomorphisms is a right orthogonal class of a family of epimorphisms between finitely presentable objects.
    From this, we can easily see that a filtered colimit of $\bT$-closed monomorphisms is again a $\bT$-closed monomorphism.
    In particular, the filtered colimit $m=(m_J)_J\colon\Colim{J\in\bI}B_J\to\prod_I A_I$ is a $\bT$-closed monomorphism, hence $\Colim{J\in\bI}B_J\in\E$.
    
    We now take the filtered colimit of $(\pi_J)_J$ as follows:
    \begin{equation*}
        \begin{tikzcd}[small]
            \prod_{I}A_I & & & \\
            & \Colim{J\in\bI}B_J\arrow[ul,"m"',hook']\arrow[rr,"\Colim{J\in\bI}\pi_J"] & & \Colim{J\in\bI}A_J \\
            B_J\arrow[uu,"m_J",hook]\arrow[ru,hook]\arrow[rr,"\pi_J"'] & & A_J\arrow[ru] &
        \end{tikzcd}\quad\quad\incat{\bT\-\PMod}.
    \end{equation*}
    Since $U^\rho$ preserves filtered colimits, by \cref{prop:locret_filtered_colim_of_retracts}, $\Colim{J\in\bI}\pi_J$ is a $U^\rho$-local retraction.
    Consequently, we have $\Colim{J\in\bI}A_J\in\E$.
\end{proof}

The above theorem gives several useful corollaries.
Taking $\rho$ as the trivial extension $\rho\colon (S,\varnothing,\varnothing)\to (S,\Sigma,\bT)$, we obtain the first corollary:
\begin{corollary}\label{cor:birkhoff_1}
    Let $(S,\Sigma,\bT)$ be a partial Horn theory such that the set $S$ of sorts is finite.
    Then, for every replete full subcategory $\E\subseteq\bT\-\PMod$, the following are equivalent:
    \begin{enumerate}
        \item
        $\E$ is definable by Horn formulas, i.e., 
        there exists a set $E$ of Horn formulas over $\Sigma$ satisfying $\E=(\bT+\bT')\-\PMod$, where $\bT':=\{ \top\seq{\tup{x}}\phi\}_{\tup{x}.\phi\in E}.$
        \item
        $\E\subseteq\bT\-\PMod$ is closed under products, $\Sigma$-closed subobjects, and surjections.
    \end{enumerate}
\end{corollary}

Taking $\rho$ as the identity $\bT\to\bT$, we obtain the second corollary:
\begin{corollary}\label{cor:birkhoff_2}
    Let $(S,\Sigma,\bT)$ be a partial Horn theory such that $\bT\-\PMod$ satisfies ACC.
    Then, for every replete full subcategory $\E\subseteq\bT\-\PMod$, the following are equivalent:
    \begin{enumerate}
        \item
        $\E$ is definable by Horn sequents, i.e., 
        there exists a set $\bT'$ of Horn sequents over $\Sigma$ satisfying $\E=(\bT+\bT')\-\PMod$.
        \item
        $\E\subseteq\bT\-\PMod$ is closed under products, $\Sigma$-closed subobjects, and local retracts.
    \end{enumerate}
\end{corollary}

Taking $\rho$ as a ``relative algebraic theory,'' we obtain the third corollary:
\begin{corollary}\label{cor:birkhoff_3}
    Let $(S,\Sigma,\bS)$ be a partial Horn theory such that $\bS\-\PMod$ satisfies ACC.
    Let $(\Omega,E)$ be an $\bS$-relative algebraic theory with the forgetful functor $U\colon\Alg(\Omega,E)\to\bS\-\PMod$.
    Then, for every replete full subcategory $\E\subseteq\Alg(\Omega,E)$, the following are equivalent:
    \begin{enumerate}
        \item
        $\E$ is definable by $\bS$-relative judgments, i.e., 
        there exists a set $E'$ of $\bS$-relative judgments satisfying $\E=\Alg(\Omega,E+E')$.
        \item
        $\E\subseteq\Alg(\Omega,E)$ is closed under products, $\Sigma$-closed subobjects, and $U$-local retracts.
    \end{enumerate}
\end{corollary}

\begin{remark}\label{rem:locret_is_essential}
    The concept of local retracts is essential for filtered colimit elimination.
    In fact, in \cref{thm:filcolim_elim}, closure under products, $\bT$-closed subobjects, and $U^\rho$-retracts does not imply closure under filtered colimits.
    If that were true, \cref{cor:birkhoff_2} must hold where local retracts are replaced with retracts.
    Of course, it is false.
    To see this, consider a single-sorted partial Horn theory $\bS$ as follows:
    \begin{gather*}
        \Sigma:=\{e\colon\text{constant}\}\cup\{ u_n\colon\text{unary function} \}_{n\in\bN},\\
        \bS:=\{ \top\seq{()}u_n(e)=e \}_{n\in\bN}.
    \end{gather*}
    An $\bS$-model is a set $M$ equipped with a constant $\intpn{e}{M}\in M$ and countably many partial endomorphisms $\intpn{u_n}{M}\colon M\pto M$ whose value on the constant $\intpn{e}{M}$ is always defined to be $\intpn{e}{M}$.
    Let $\E\subseteq\bS\-\PMod$ be the full subcategory defined by the following infinitary sequent:
    \begin{equation*}
        \left( \bigwedge_{n\in\bN} u_n(x)=e \right) \seq{x} x=e.
    \end{equation*}
    We now claim that $\E\subseteq\bS\-\PMod$ is closed under products, $\Sigma$-closed subobjects, and retracts but not closed under filtered colimits.
    To show that $\E\subseteq\bS\-\PMod$ is not closed under filtered colimits, for each $n\in\bN$, define $A_n:=\{ 0, a\}\in\bS\-\PMod$ by the following:
    \begin{gather*}
        \intpn{e}{A_n}:=0, \\
        \intpn{u_k}{A_n}(a):=
        \begin{cases*}
            0 & if $k<n$, \\
            \text{undefined} & if $k\ge n$.
        \end{cases*}
    \end{gather*}
    Now, the identity map on $\{0,a\}$ yields an $\bbomega$-chain $A_0\to A_1\to A_2\to\cdots$ in $\bS\-\PMod$.
    Then, the colimit $A:=\Colim{n\in\bbomega}A_n$ is described as follows:
    \begin{gather*}
        A:=\{0, a\},\quad
        \intpn{e}{A}:=0,\quad
        \intpn{u_k}{A}(a):=0~(\forall k).
    \end{gather*}
    In particular, the filtered colimit $A$ does not belong to $\E$.
\end{remark}

\begin{example}[Finite-sorted algebras]
    Considering a partial Horn theory $(S,\varnothing,\varnothing)$ such that $S$ is finite, 
    we get the Birkhoff's theorem in finite sorts \cite{adamek2012birkhoffs} as a corollary of \cref{cor:birkhoff_3}.
    Let $(\Omega,E)$ be an $S$-sorted algebraic theory.
    Then, a replete full subcategory $\E\subseteq\Alg(\Omega,E)$ is definable by equations if and only if it is closed under products, subobjects, and surjections.
\end{example}

\begin{remark}\label{rem:ACC_cannot_be_removed}
    The assumption that $\bS\-\PMod$ satisfies ACC cannot be removed from \cref{thm:filcolim_elim} even if $\bS$ is finite-sorted.
    To show this, let $\bS$ be the single-sorted partial Horn theory for sets with countably many constants as follows:
    \begin{gather*}
        \Sigma:=\{ c_n\colon\text{constant} \}_{n\in\bN},\quad \bS:=\{ \top\seq{()}c_n\defined \}_{n\in\bN}.
    \end{gather*}
    Consider the full subcategory of $\bS\-\PMod$
    \[
    \E:=\{1\}\cup\{ M\in\bS\-\PMod \mid \exists i,j\text{~s.t.}\intpn{c_i}{M}\neq\intpn{c_j}{M} \},
    \]
    where $1$ is the terminal.
    A local retraction in the category $\bS\-\PMod$ is precisely a surjection that does not merge any constants.
    Thus, we can see that $\E\subseteq\bS\-\PMod$ is closed under products, $\Sigma$-closed subobjects, and local retracts.
    We now show that it is not closed under filtered colimits.
    For each $n\in\bN$, define $A_n:=\bN\cup\{\infty\}\in\bS\-\PMod$ by $\intpn{c_i}{A_n}:=\max(i-n,0)$.
    Let $f_n\colon A_n\to A_{n+1}$ be the morphism
    \begin{equation*}
        f_n(x):=
        \begin{cases*}
            \max(x-1,0) & if $x\neq\infty$, \\
            \infty & if $x=\infty$.
        \end{cases*}
    \end{equation*}
    Then, the filtered colimit $A_\infty$ of $A_0\arr{f_0}A_1\arr{f_1}\cdots$ consists of two different points, and all constants are merged there, hence $A_\infty\not\in\E$ even though $A_n\in\E$.
\end{remark}

\begin{example}[Ordered algebras]
    Considering the partial Horn theory $\bS_\pos$ for posets as in \cref{eg:pht_for_posets}, we get the Birkhoff's theorem for ordered algebras as a consequence of \cref{cor:birkhoff_3}.
    Let $(\Omega,E)$ be an $\bS_\pos$-relative algebraic theory and let $U\colon\Alg(\Omega,E)\to\Pos$ be the forgetful functor.
    Then, a replete full subcategory $\E\subseteq\Alg(\Omega,E)$ is definable by $\bS_\pos$-relative judgments if and only if it is closed under products, $\Sigma_\pos$-closed subobjects, and $U$-local retracts.

    However, there is another version of the Birkhoff's theorem for ordered algebras.
    Given an $\bS_\pos$-relative algebraic theory $(\Omega,E)$, we have the associated partial Horn theory $\pht{\Omega}{E}$ as in \cref{def:pht_for_rat}.
    Applying \cref{cor:birkhoff_1} to $\pht{\Omega}{E}$, we get the following theorem:
    A replete full subcategory $\E\subseteq\Alg(\Omega,E)$ is definable by inequalities (Horn formulas over $\Sigma_\pos+\Omega$) if and only if it is closed under products, $\Sigma_\pos$-closed subobjects, and surjections.
    Note that a $\pht{\Omega}{E}$-closed subobject is simply an $\Sigma_\pos$-closed subobject.
    This theorem is a direct generalization of a classical Birkhoff's theorem for ordered algebras in \cite{bloom1976varieties}.
    This is because a morphism $f\colon\bA\to\bB$ between ordered algebras is surjective if and only if $f$ can be written as a quotient $\bA\to\bA/\theta$ of a congruence $\theta$ determined by an \emph{admissible preorder} on $\bA$ in the sense of \cite{bloom1976varieties}.
\end{example}

By filtered colimit elimination, we can reformulate Birkhoff's theorem for relative algebras in the HSP-type form (\cref{cor:HSP-type_formalization}).

\begin{notation}
    Let $\rho\colon \bS\to\bT$ be a theory morphism between partial Horn theories.
    For each replete full subcategory $\E\subseteq\bT\-\PMod$, let $\cloprod(\E)\subseteq\bT\-\PMod$ denote the full subcategory consisting of all products of objects in $\E$.
    Similarly, let $\clocsub{\bT}(\E)$ denote the full subcategory consisting of all $\bT$-closed subobjects of objects in $\E$, and let $\clolret{\rho}(\E)$ denote the full subcategory consisting of all $U^\rho$-local retracts of objects in $\E$.
\end{notation}

\begin{lemma}\label{lem:HSP_operator}
    Let $\rho\colon \bS\to\bT$ be a theory morphism between partial Horn theories and let $\E\subseteq\bT\-\PMod$ be a replete full subcategory.
    \begin{enumerate}
        \item\label{lem:HSP_operator-1}
        The closure operators $\cloprod$, $\clocsub{\bT}$, and $\clolret{\rho}$ are idempotent, i.e.,
        \begin{gather*}
            \cloprod\cloprod(\E)=\cloprod(\E),\quad
            \clocsub{\bT}\clocsub{\bT}(\E)=\clocsub{\bT}(\E),\quad
            \clolret{\rho}\clolret{\rho}(\E)=\clolret{\rho}(\E).
        \end{gather*}
        \item\label{lem:HSP_operator-2}
        $\cloprod\clolret{\rho}(\E)\subseteq\clolret{\rho}\cloprod(\E)$ holds.
        \item\label{lem:HSP_operator-3}
        $\cloprod\clocsub{\bT}(\E)\subseteq\clocsub{\bT}\cloprod(\E)$ holds.
        \item\label{lem:HSP_operator-4}
        $\clocsub{\bT}\clolret{\rho}(\E)\subseteq\clolret{\rho}\clocsub{\bT}(\E)$ holds.
    \end{enumerate}
    In particular, $\clolret{\rho}\clocsub{\bT}\cloprod(\E)\subseteq\bT\-\PMod$ is the smallest replete full subcategory containing $\E$ which is closed under products, $\bT$-closed subobjects, and $U^\rho$-local retracts.
\end{lemma}
\begin{proof}
    \ref{lem:HSP_operator-1} is trivial.
    \ref{lem:HSP_operator-2} follows from the fact that a product of $U^\rho$-local retractions is again a $U^\rho$-local retraction.
    Since the class of all $\bT$-closed monomorphisms is the right class of an orthogonal factorization system \cite[Theorem 4.14]{kawase2023birkhoffs}, a product of $\bT$-closed monomorphisms is again a $\bT$-closed monomorphism.
    Thus, \ref{lem:HSP_operator-3} follows.
    Finally, \ref{lem:HSP_operator-4} follows from the fact that $\bT$-closed monomorphisms and local retractions are stable under pullback.
\end{proof}

\begin{corollary}\label{cor:HSP-type_formalization}
    Let $\rho\colon\bS\to\bT$ be a theory morphism between partial Horn theories and assume that the category $\bS\-\PMod$ satisfies ACC.
    Then, a replete full subcategory $\E\subseteq\bT\-\PMod$ is definable by $\rho$-relative judgments if and only if $\clolret{\rho}\clocsub{\bT}\cloprod(\E)=\E$ holds.
\end{corollary}

We now give a weak converse of filtered colimit elimination.

\begin{definition}
    We say a partial Horn theory $\bS$ admits \emph{filtered colimit elimination} if \cref{thm:filcolim_elim} holds, i.e., 
    for every theory morphism $\rho\colon \bS\to\bT$, a replete full subcategory $\E\subseteq\bT\-\PMod$ is closed under filtered colimits whenever it is closed under products, $\bT$-closed subobjects, and $U^\rho$-local retracts.
\end{definition}

\begin{theorem}\label{thm:converse_fil_colim_elim}
    Let $\bS$ be a partial Horn theory that admits filtered colimit elimination.
    \begin{enumerate}
        \item\label{thm:converse_fil_colim_elim-1}
        The full subcategory $\fpconn{\bS\-\PMod}\subseteq\bS\-\PMod$ consisting of all finitely presentable connected objects satisfies the ascending chain condition (see \cref{def:connected_object}).
        \item\label{thm:converse_fil_colim_elim-2}
        Assume that the morphism from the initial to the terminal is a strong monomorphism in $\bS\-\PMod$.
        Equivalently, assume that every non-initial object has at least two parallel morphisms from itself.
        Then, the full subcategory $\fp{\bS\-\PMod}\subseteq\bS\-\PMod$ consisting of all finitely presentable objects satisfies the ascending chain condition.
    \end{enumerate}
\end{theorem}
\begin{proof}
    The proofs of both \ref{thm:converse_fil_colim_elim-1} and \ref{thm:converse_fil_colim_elim-2} can be discussed in parallel.
    Let $\A:=\bS\-\PMod$ and take an $\bbomega$-chain $A_0\to A_1\to A_2\to\cdots$ of finitely presentable objects in $\A$.
    We have to show that there exists $N\in\bN$ such that $(A_n)_{n\ge N}$ are strongly connected to each other.
    To show this, take the colimit $A:=\Colim{n\in\bbomega}A_n$ of the $\bbomega$-chain.
    If $A$ is initial, it turns out that all $A_n~(n\in\bN)$ are strongly connected to each other.
    Thus, in what follows, we assume that $A$ is not initial.

    We now claim that the coproduct $A+A$ is not a subterminal in both cases \ref{thm:converse_fil_colim_elim-1} and \ref{thm:converse_fil_colim_elim-2}.
    In the case of \ref{thm:converse_fil_colim_elim-1}, we can additionally assume that all $A_n$ are connected.
    By \cref{prop:conn_colim_of_conn_obj}, $A$ is connected.
    Then, $\A(A,A+A)\cong\A(A,A)+\A(A,A)$ implies that there are at least two morphisms from $A$ to $A+A$, hence $A+A$ is not a subterminal.
    We now turn to the case of \ref{thm:converse_fil_colim_elim-2}.
    Since $A$ is not initial, by assumption, there are different parallel morphisms $f,g\colon A\to A'$ for some $A'$.
    Considering post-compositions with the canonical morphism $(f,g)\colon A+A\to A'$, we can see that the two coprojections $A\rightarrow A+A\leftarrow A$ do not coincide, hence $A+A$ is not a subterminal.

    Consider the following replete full subcategory $\E\subseteq\A$:
    \begin{equation*}
        \E:=\{ X \mid \text{There is a morphism }X\to A_n\text{ for some }n\in\bN \}.
    \end{equation*}
    Let $\rho\colon\bS\to\bS$ be the identity theory morphism.
    Now, we have the following:
    \begin{gather}\label{eq:HSP_E}
        \clolret{\rho}\clocsub{\bS}\cloprod(\E)
        =\clolret{\rho}\clocsub{\bS}(\{1\}\cup\E)
        =\clolret{\rho}(\clocsub{\bS}(1)\cup\E)
        =\clolret{\rho}\clocsub{\bS}(1) \cup \clolret{\rho}(\E)
        =\clocsub{\bS}(1) \cup \clolret{\rho}(\E).
    \end{gather}
    We now explain the last equality $\clolret{\rho}\clocsub{\bS}(1)=\clocsub{\bS}(1)$ in \cref{eq:HSP_E}.
    If $B\in\clolret{\rho}\clocsub{\bS}(1)$, we get an $\bS$-closed monomorphism $m\colon X\hookrightarrow 1$ to the terminal and a local retraction $p\colon X\to B$.
    By \cref{prop:locret_implies_reg.epi}, $p$ is a strong epimorphism.
    Thus, the following square has a unique diagonal filler:
    \begin{equation*}
        \begin{tikzcd}
            X\arrow[d,"p"]\arrow[r,equal] & X\arrow[d,hook,"m"] \\
            B\arrow[ur,dashed]\arrow[r,"!"'] & 1
        \end{tikzcd}\incat{\A}.
    \end{equation*}
    Therefore, $p$ is an isomorphism, in particular, $B$ is an $\bS$-closed subobject of the terminal.
    
    Since $\bS$ admits filtered colimit elimination, the filtered colimit $A+A\cong\Colim{n\in\bbomega}(A_n+A_n)$ belongs to $\clolret{\rho}\clocsub{\bS}\cloprod(\E)$.
    We already know that $A+A$ is not a subterminal, hence \cref{eq:HSP_E} implies $A+A\in\clolret{\rho}(\E)$.
    Thus, we get a local retraction $q\colon Y\to A+A$ and a morphism $Y\to A_N$ for some $N$.
    Fix $n\ge N$.
    Since $A_n$ is finitely presentable and $q$ is a local retraction, there is a morphism $A_n\to Y$:
    \begin{equation*}
        \begin{tikzcd}
            && Y\arrow[d,"q"]\arrow[r] & A_N \\
            A_n\arrow[urr,dashed]\arrow[r,"i"']& A\arrow[r,"j"'] & A+A &
        \end{tikzcd}\quad\quad\incat{\A},
    \end{equation*}
    where $i$ and $j$ are coprojections.
    We now get a morphism $A_n\to A_N$, hence $(A_n)_{n\ge N}$ are strongly connected to each other.
    This completes the proof.
\end{proof}

\begin{remark}
    At least, \cref{thm:converse_fil_colim_elim}\ref{thm:converse_fil_colim_elim-1} is not the true converse of \cref{thm:filcolim_elim}.
    Indeed, the partial Horn theory $\bS$ of \cref{rem:ACC_cannot_be_removed} does not admit filtered colimit elimination, but it satisfies ACC for connected objects since $A_0$ in \cref{rem:ACC_cannot_be_removed} is the unique connected object.
    On the other hand, we do not know any example of a partial Horn theory that does not admit filtered colimit elimination and that satisfies ACC for finitely presentable objects.
\end{remark}

\section{Computation of strongly connected components}\label{section6}
In \cref{subsection6.1}, we compute strongly connected components in \emph{locally connected categories}, which include all presheaf categories.
In \cref{subsection6.2}, we restrict our attention to more special case, categories of topological group actions on a set.

\subsection{Locally connected categories}\label{subsection6.1}

\begin{definition}\label{def:connected_object}
    An object $C\in\C$ in a locally small category $\C$ is \emph{connected} if its hom-functor $\C(C,\bullet)\colon\C\to\Set$ preserves small coproducts.
    Equivalently, an object $C$ is connected if every morphism $C\to\coprod_{i}X_i$ into a small coproduct factors uniquely through a unique coprojection $X_i\to X$.
\end{definition}

\begin{proposition}\label{prop:conn_colim_of_conn_obj}
    Every connected colimit of connected objects is a connected object again.
\end{proposition}
\begin{proof}
    Let $C:=\Colim{I\in\bI}C_I$ be a connected colimit in a category $\C$ of connected objects, i.e., $\bI$ is a connected category and all $C_I\in\C$ are connected.
    Let $X=\coprod_{J}X_J$ be a small coproduct in $\C$.
    Since coproducts commute with connected limits in $\Set$, we have the following:
    \begin{equation*}
        \C(C,X)=\Lim{I\in\bI^\op}\C(C_I,X)
        \cong\Lim{I\in\bI^\op}\coprod_J \C(C_I,X_J)
        \cong\coprod_J \Lim{I\in\bI^\op}\C(C_I,X_J)
        \cong\coprod_J \C(C,X_J).
    \end{equation*}
    This proves that $C$ is connected.
\end{proof}

\begin{definition}
    A category $\C$ is \emph{locally connected} if it has small coproducts and every object is a small coproduct of connected objects.
\end{definition}

\begin{remark}
    An object $C\in\E$ in a topos $\E$ is connected if and only if its hom-functor $\E(C,\bullet)\colon\E\to\Set$ preserves binary coproducts. Therefore, a Grothendieck topos is locally connected if and only if it is a locally connected topos in the usual sense, which is also called a molecular topos in \cite{barrpare1980molecular}.
    In particular, every presheaf category is locally connected.
\end{remark}

\begin{notation}
    Let $\conn{\C}\subseteq\C$ denote the full subcategory consisting of all connected objects in $\C$.
\end{notation}

\begin{definition}[The category of families]
    Given a category $\A$, we define the \emph{category of families of $\A$}, denoted by $\Fam{\A}$:
    An object in $\Fam{\A}$ is a small family $(A_i)_{i\in I}$ of objects in $\A$, 
    and a morphism $(A_i)_{i\in I}\to (B_j)_{j\in J}$ is a map $f\colon I\to J$ together with a family \mbox{$(f_i\colon A_i\to B_{f(i)})_{i\in I}$} of morphisms in $\A$.
    Note that $\Fam{\A}$ is the free coproduct cocompletion of $\A$.
\end{definition}

The categories of families characterize locally connected categories:
\begin{theorem}[{\cite[Lemma 42]{carbonivitale1998}}]\label{thm:loc-conn-characterization}
    For every category $\C$, the following are equivalent:
    \begin{enumerate}
        \item\label{thm:loc-conn-characterization-1}
        $\C$ is locally connected.
        \item\label{thm:loc-conn-characterization-2}
        There exists a category $\A$ such that $\C\simeq\Fam{\A}$.
    \end{enumerate}
    The category $\A$ in the condititon \ref{thm:loc-conn-characterization-2} can be chosen as $\conn{\C}$.
\end{theorem}
\begin{proof}
    {[\ref{thm:loc-conn-characterization-2}$\implies$\ref{thm:loc-conn-characterization-1}]}
    We claim that $\Fam{\A}$ has small coproducts.
    Indeed, for a family $\{ (A_{\lambda,i})_{i\in I_\lambda} \}_{\lambda\in\Lambda}$ of objects in $\Fam{\A}$, we can see that $(A_{\lambda,i})_{(\lambda,i)\in \coprod_{\lambda\in\Lambda}I_\lambda}$ is a coproduct of them.
    Furthermore, every single-indexed object $(A_i)_{i\in\{i\}}$ is connected in $\Fam{\A}$, 
    and every object in $\Fam{\A}$ is a coproduct of single-indexed objects, hence $\Fam{\A}$ is locally connected.

    {[\ref{thm:loc-conn-characterization-1}$\implies$\ref{thm:loc-conn-characterization-2}]}
    Since $\C$ has coproducts, we can define a functor $K\colon\Fam{\conn{\C}}\to\C$ by $(A_i)_{i\in I}\mapsto \coprod_{i\in I} A_i$.
    By definition of locally connected categories, the functor $K$ is essentially surjective.
    Moreover, $K$ is fully faithful.
    Indeed, given a morphism 
    \[
    K((A_i)_{i\in I})=\coprod_{i\in I}A_i \longarr{(f_i)} \coprod_{j\in J}B_j=K((B_j)_{j\in J}) \quad\quad\incat{\C},
    \]
    we can see that each component $f_i\colon A_i\to\coprod_{j\in J} B_j$ factors through a unique coprojection $B_j\to\coprod_{j\in J} B_j$, and this yields a morphism $(A_i)_{i\in I}\to (B_j)_{j\in J}$ in $\Fam{\conn{\C}}$.
\end{proof}

\begin{definition}
    A lower class $L\subseteq\A$ in a category $\A$ is \emph{essentially small} [resp. \emph{finitely generated}] if there is a small [resp. finite] set $S\subseteq\ob\A$ such that $\lowset{S} =L$, where $\lowset{S}$ is the smallest lower class containing $S$.
    Let $\eslowlat{\A}\subseteq\lowlat{\A}$ [resp. $\fglowlat{\A}\subseteq\lowlat{\A}$] denote the (large) subposet consisting of all essentially small [resp. finitely generated] lower classes.
\end{definition}

\begin{theorem}\label{thm:srg-conn-of-Fam}
    For a category $\A$, there are isomorphisms $\sigma(\Fam{\A})\cong\eslowlat{\A}\cong\eslowlat{\sigma(\A)}$ of (large) posets.
    In particular, by \cref{thm:loc-conn-characterization}, we get an isomorphism $\sigma(\C)\cong\eslowlat{\sigma(\conn{\C})}$ for every locally connected category $\C$.
\end{theorem}
\begin{proof}
    The second isomorphism $\eslowlat{\A}\cong\eslowlat{\sigma(\A)}$ is trivial.
    To get the first isomorphism $\sigma(\Fam{\A})\cong\eslowlat{\A}$, we define a functor $K\colon \Fam{\A}\to\eslowlat{\A}$ by the following:
    \[
    (A_i)_{i\in I}\quad\overset{K}{\mapsto}\quad \lowset{\{ A_i \mid i\in I \}} = \{ X\in\A \mid \text{There is a morphism }X\to A_i\text{ for some }i\in I \}.
    \]
    We can see that the functor $K$ is surjective on objects and full.
    Thus, it follows from \cref{prop:full-esssurj-srgconn} that $\sigma(\Fam{\A})\cong\sigma(\eslowlat{\A})\cong\eslowlat{\A}$.    
\end{proof}

\begin{example}
    We follow the notation used in \cref{eg:srg_conn}\ref{eg:srg_conn-cospan}.
    The category $\Cospan$ is locally connected since it is a presheaf category.
    The poset $\sigma(\conn{\Cospan})$ is displayed as follows:
    
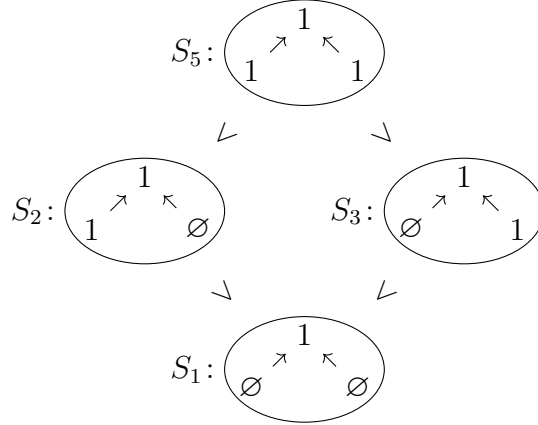
\begin{equation*}
        \begin{tikzpicture}[scale=0.7]
            \node (a1) at (-1,3) {1};
            \node (a) at (0,4) {1};
            \node (a2) at (1,3) {1};
            \node (a*) at (0,3.35) {};
            \node at (-2,3.35) {$S_5\colon$};
            \draw[->](a1)--(a);
            \draw[->](a2)--(a);
            \node (b1) at (-4,0) {1};
            \node (b) at (-3,1) {1};
            \node (b2) at (-2,0) {$\varnothing$};
            \node (b*) at (-3,0.35) {};
            \node at (-5,0.35) {$S_2\colon$};
            \draw[->](b1)--(b);
            \draw[->](b2)--(b);
            \node (c1) at (2,0) {$\varnothing$};
            \node (c) at (3,1) {1};
            \node (c2) at (4,0) {1};
            \node (c*) at (3,0.35) {};
            \node at (1,0.35) {$S_3\colon$};
            \draw[->](c1)--(c);
            \draw[->](c2)--(c);
            \node (d1) at (-1,-3) {$\varnothing$};
            \node (d) at (0,-2) {1};
            \node (d2) at (1,-3) {$\varnothing$};
            \node (d*) at (0,-2.65) {};
            \node at (-2,-2.65) {$S_1\colon$};
            \draw[->](d1)--(d);
            \draw[->](d2)--(d);
            \draw (a*) circle[x radius=1.5,y radius=1];
            \draw (b*) circle[x radius=1.5,y radius=1];
            \draw (c*) circle[x radius=1.5,y radius=1];
            \draw (d*) circle[x radius=1.5,y radius=1];
            \node[rotate=45] at (-1.5,1.8) {$<$};
            \node[rotate=-45] at (1.5,1.8) {$>$};
            \node[rotate=-45] at (-1.5,-1.2) {$>$};
            \node[rotate=45] at (1.5,-1.2) {$<$};
        \end{tikzpicture}
    \end{equation*}
    In this case, $\sigma(\conn{\Cospan})$ is small, hence $\sigma(\Cospan)\cong\lowlat{\sigma(\conn{\Cospan})}$ and this isomorphism is displayed as follows:
    \begin{equation*}
        \begin{tikzcd}[row sep=tiny]
            & S_5 & \\
            & S_4\arrow[u,phantom,"<"sloped] & \\
            S_2\arrow[ur,phantom,"<"sloped] & & S_3\arrow[ul,phantom,">"sloped] \\
            & S_1\arrow[ul,phantom,">"sloped]\arrow[ur,phantom,"<"sloped] & \\
            & S_0\arrow[u,phantom,"<"sloped] &
        \end{tikzcd}
        \quad\quad\cong\quad\quad
        \begin{tikzcd}[tiny]
            & \lowset{S_5} & \\
            & \lowset{\{S_2,S_3\}}\arrow[u,phantom,"\subset"sloped] & \\
            \lowset{S_2}\arrow[ur,phantom,"\subset"sloped] & & \lowset{S_3}\arrow[ul,phantom,"\supset"sloped] \\
            & \lowset{S_1}\arrow[ul,phantom,"\supset"sloped]\arrow[ur,phantom,"\subset"sloped] & \\
            & \varnothing\arrow[u,phantom,"\subset"sloped] &
        \end{tikzcd}
    \end{equation*}
\end{example}

\begin{lemma}\label{lem:LX_satisfies_ACC}
    Let $X\in\POS$ be a large poset.
    Then, the following are equivalent:
    \begin{enumerate}
        \item\label{lem:LX_satisfies_ACC-1}
        The large poset $\lowlat{X}$ satisfies ACC.
        \item\label{lem:LX_satisfies_ACC-2}
        The large poset $\eslowlat{X}$ satisfies ACC.
        \item\label{lem:LX_satisfies_ACC-3}
        The large poset $\fglowlat{X}$ satisfies ACC.
        \item\label{lem:LX_satisfies_ACC-4}
        Every lower class in $X$ is finitely generated.
        \item\label{lem:LX_satisfies_ACC-5}
        Let $X=(X,\lowlat{X})$ be the large topological space whose underlying class is $X$ and whose open sets are the lower classes in $X$.
        Then, every open subspace of $X$ is \emph{large-compact}, i.e., every open cover indexed by a class has a finite subcover.
        \item\label{lem:LX_satisfies_ACC-6}
        For every sequence $(x_n)_n$ of elements of $X$, there exists $N\in\bN$ such that for every $n\ge N$, $x_m\ge x_n$ holds for some $m<n$.
    \end{enumerate}
    If $X$ is totally ordered, these conditions are equivalent to that $X$ satisfies ACC.
\end{lemma}
\begin{proof}
    {[\ref{lem:LX_satisfies_ACC-1}$\implies$\ref{lem:LX_satisfies_ACC-2}$\implies$\ref{lem:LX_satisfies_ACC-3}]} Trivial.

    {[\ref{lem:LX_satisfies_ACC-3}$\implies$\ref{lem:LX_satisfies_ACC-4}]}
    By \ref{lem:LX_satisfies_ACC-3}, every non-empty class of finitely generated lower classes contains a maximal element.
    Thus, every non-empty lower class $L\subseteq X$ has a maximal finitely generated lower subclass $M\subseteq L$.
    Then we have $L=M$, hence $L$ is finitely generated.

    {[\ref{lem:LX_satisfies_ACC-4}$\implies$\ref{lem:LX_satisfies_ACC-5}]}
    Let $L\subseteq X$ be a lower class and let $L=\bigcup_{i\in I}L_i$, where all $L_i$ are lower classes and $I$ is a class.
    By \ref{lem:LX_satisfies_ACC-4}, $L$ is generated by finitely many elements $x_1,\dots,x_n\in X$.
    For each $1\le k\le n$, we can choose $i_k\in I$ such that $x_k\in L_{i_k}$.
    Then, we have $L=\bigcup_{1\le k\le n}L_{i_k}$.

    {[\ref{lem:LX_satisfies_ACC-5}$\implies$\ref{lem:LX_satisfies_ACC-6}]}
    Let $(x_n)_n$ be a sequence of elements of $X$.
    Consider the lower class
    \[
    L:=\lowset{\{x_n\mid n\in\bN\}}=\bigcup_{n\in\bN}\lowset{\{ x_k \mid 0\le k\le n \}}.
    \]
    By \ref{lem:LX_satisfies_ACC-5}, there is $N\in\bN$ such that $L=\lowset{\{ x_k \mid 0\le k\le N \}}$, which satisfies the condition of \ref{lem:LX_satisfies_ACC-6}.

    {[\ref{lem:LX_satisfies_ACC-6}$\implies$\ref{lem:LX_satisfies_ACC-1}]}
    Assume that $\lowlat{X}$ does not satisfy ACC.
    Then, there is a strictly ascending sequence $L_0\subsetneq L_1\subsetneq L_2\subsetneq\cdots$ of lower classes in $X$.
    Choosing an element $x_n\in L_{n+1}\setminus L_n$ for each $n$, we can see that the sequence $(x_n)_n$ does not satisfy the condition of \ref{lem:LX_satisfies_ACC-6}.
\end{proof}

\begin{remark}
    A topological space is called \emph{noetherian} if the open subsets satisfy the ascending chain condition.
    Thus, a poset $X$ satisfies the conditions of \cref{lem:LX_satisfies_ACC} if and only if $X$ is a noetherian topological space with the lower set topology.
\end{remark}

\begin{corollary}\label{cor:loc.conn.cat_ACC}
    A locally connected category $\C$ satisfies the ascending chain condition if and only if every lower class in $\conn{\C}$ is finitely generated.
\end{corollary}

\subsection{Categories of group actions}\label{subsection6.2}
We recall the definition of atomic Grothendieck toposes:

\begin{definition}\quad
    \begin{enumerate}
        \item
        An object $A\in\E$ in a Grothendieck topos $\E$ is called an \emph{atom} if $A$ is non-zero and its only subobjects are $A$ and $0$.
        Let $\atom{\E}\subseteq\E$ denote the full subcategory consisting of all atoms in $\E$.
        \item
        A Grothendieck topos $\E$ is \emph{atomic} if every object in $\E$ is a disjoint union of atoms.
    \end{enumerate}
\end{definition}

\begin{remark}\label{rem:srg-conn-of-atomic-topos}
    An object in an atomic Grothendieck topos $\E$ is connected if and only if it is an atom.
    In particular, we have an isomorphism $\sigma(\E)\cong\eslowlat{\sigma(\atom{\E})}$ as (large) posets.
\end{remark}

\begin{definition}
    Let $G$ be a topological group.
    We define a category $\opensub{G}$ as follows:
    \begin{itemize}
        \item
        An object in $\opensub{G}$ is an open subgroup of $G$.
        \item
        A morphism $H\to K$ in $\opensub{G}$ is an element $g\in G$ such that $gHg^{-1}\subseteq K$.
        \item
        The composition of $g\colon H\to K$ and $g'\colon K\to L$ is defined as $g'g\in G$.
    \end{itemize}
\end{definition}

\begin{lemma}\label{lem:atom_is_opensub}
    Let $G\-\Set$ be the category of continuous actions of a topological group $G$ on a set.
    Then, there is an equivalence $\atom{G\-\Set}\simeq\opensub{G}$.
\end{lemma}
\begin{proof}
    See \cite[C3.5.9(b)]{johnstone2002sketches}.
\end{proof}

\begin{theorem}\label{thm:srg_conn_of_G-Set}
    Let $G$ be a topological group.
    \begin{enumerate}
        \item\label{thm:srg_conn_of_G-Set-1}
        The category $G\-\Set$ of continuous actions of $G$ satisfies ACC if and only if every lower set in $\opensub{G}$ is finitely generated.
        \item\label{thm:srg_conn_of_G-Set-2}
        There is an isomorphism $\sigma(G\-\Set)\cong\lowlat{\sigma(\opensub{G})}$ as posets.
        In particular, the class of all strongly connected components in $G\-\Set$ is small.
    \end{enumerate}
\end{theorem}
\begin{proof}
    These follow from \cref{cor:loc.conn.cat_ACC,rem:srg-conn-of-atomic-topos,lem:atom_is_opensub}.
\end{proof}

\begin{example}
    Let $\bZ$ be the (discrete) group of integers.
    Since $\bZ$ is abelian, the morphisms in $\opensub{\bZ}$ are simply the inclusions.
    Thus, we have $\sigma(\bZ\-\Set)\cong\lowlat{\sub{\bZ}}$, where $\sub{\bZ}$ is the poset of all subgroups of $\bZ$ with the inclusions.
    Moreover, $\bZ\-\Set$ does not satisfy ACC since the lowerset $\lowset{\{ p\bZ \mid p\colon\text{a prime number} \}}$ in $\sub{\bZ}$ is not finitely generated.
\end{example}

\begin{example}
    Let $G:=\aut{\bN}$ be the topological group of all bijections on $\bN$ whose topology is induced from the product topology on $\bN^\bN$.
    A continuous action of $G$ is called a \emph{nominal set}, and the category $G\-\Set$ is equivalent to the \emph{Schanuel topos}.
    We now show that the category $G\-\Set$ does not satisfy ACC.
    To show this, consider the open subgroups
    \[
    H_n:=\{ \sigma\in G \mid \forall k<n,~\sigma(k)<n \}
    \]
    of $G$ for each $n\ge 1$.
    Then, it follows that for each open subgroup $K\subseteq G$, there is a morphism $K\to H_n$ in $\opensub{G}$ if and only if there exists a subset $A\subset\bN$ of order $n$ satisfying $\sigma(A)\subseteq A$ for every $\sigma\in K$.
    In particular, there is no morphism $H_n\to H_m$ in $\opensub{G}$ whenever $n>m$.
    Therefore, the lower set $\lowset{\{ H_n \mid n\ge 1 \}}$ in $\opensub{G}$ is not finitely generated.
\end{example}

\printbibliography

@book{adamek1994locally,
    AUTHOR = {Ad\'{a}mek, Ji\v{r}\'{\i} and Rosick\'{y}, Ji\v{r}\'{\i}},
     TITLE = {Locally Presentable and Accessible Categories},
    SERIES = {London Mathematical Society Lecture Note Series},
    VOLUME = {189},
 %PUBLISHER = {Cambridge University Press, Cambridge},
      YEAR = {1994},
     %PAGES = {xiv+316},
}

@misc{kawase2023birkhoffs,
      title={Birkhoff's variety theorem for relative algebraic theories}, 
      author={Yuto Kawase},
      year={2023},
      eprint={2304.04382},
      archivePrefix={arXiv},
      primaryClass={math.CT}
}

@article {bloom1976varieties,
    AUTHOR = {Bloom, Stephen L.},
     TITLE = {Varieties of ordered algebras},
   JOURNAL = {J. Comput. System Sci.},
    VOLUME = {13},
      YEAR = {1976},
    NUMBER = {2},
     PAGES = {200--212},
}

@article {adamek2012birkhoffs,
    AUTHOR = {Ad\'{a}mek, Ji\v{r}\'{\i} and Rosick\'{y}, Ji\v{r}\'{\i} and Vitale, Enrico M.},
     TITLE = {Birkhoff's variety theorem in many sorts},
   JOURNAL = {Algebra Universalis},
    VOLUME = {68},
      YEAR = {2012},
    NUMBER = {1-2},
     PAGES = {39--42},
}

@book{johnstone2002sketches,
    AUTHOR = {Johnstone, Peter T.},
     TITLE = {Sketches of an Elephant: A Topos Theory Compendium. {V}ol. 2},
    SERIES = {Oxford Logic Guides},
    VOLUME = {44},
 %PUBLISHER = {The Clarendon Press, Oxford University Press, Oxford},
      YEAR = {2002},
     %PAGES = {i--xxii, 469--1089 and I1--I71},
}

@article {barrpare1980molecular,
    AUTHOR = {Barr, Michael and Par\'{e}, Robert},
     TITLE = {Molecular toposes},
   JOURNAL = {J. Pure Appl. Algebra},
    VOLUME = {17},
      YEAR = {1980},
    NUMBER = {2},
     PAGES = {127--152},
}

@article{birkhoff1935structure,
    author={Birkhoff, Garrett},
    title={On the structure of abstract algebras},
    journal={Mathematical Proceedings of the Cambridge Philosophical Society},
    volume={31},
    year={1935},
    number={4},
    pages={433--454},
}

@incollection{linton1969outline,
    AUTHOR = {Linton, F. E. J.},
     TITLE = {An outline of functorial semantics},
 BOOKTITLE = {Sem. on {T}riples and {C}ategorical {H}omology {T}heory ({ETH}, {Z}\"{u}rich, 1966/67)},
     PAGES = {7--52},
 %PUBLISHER = {Springer, Berlin},
      YEAR = {1969},
}

@article{palmgren2007partial,
    AUTHOR = {Palmgren, E. and Vickers, S. J.},
     TITLE = {Partial Horn logic and cartesian categories},
   JOURNAL = {Ann. Pure Appl. Logic},
    VOLUME = {145},
      YEAR = {2007},
    NUMBER = {3},
     PAGES = {314--353},
}

@article{carbonivitale1998,
    AUTHOR = {Carboni, A. and Vitale, E. M.},
     TITLE = {Regular and exact completions},
   JOURNAL = {J. Pure Appl. Algebra},
    VOLUME = {125},
      YEAR = {1998},
    NUMBER = {1-3},
     PAGES = {79--116},
}

@book{golan2003semiring,
    AUTHOR = {Golan, Jonathan S.},
     TITLE = {Semirings and Affine Equations over Them: Theory and Applications},
    SERIES = {Mathematics and its Applications},
    VOLUME = {556},
 %PUBLISHER = {Kluwer Academic Publishers Group, Dordrecht},
      YEAR = {2003},
     %PAGES = {xiv+241},
}

@article{adamek2004purequotient,
    AUTHOR = {Ad\'{a}mek, J. and Rosick\'{y}, J.},
     TITLE = {On pure quotients and pure subobjects},
   JOURNAL = {Czechoslovak Math. J.},
    VOLUME = {54(129)},
      YEAR = {2004},
    NUMBER = {3},
     PAGES = {623--636},
}

\begin{table}[b]
    \centering
    \rowcolors{2}{}{gray!25}
    \begin{tabular}{cccc}
        \hline
        \multrow{Locally finitely \\ presentable \\ category} & Objects & \multrow{Number of \\ strongly connected \\ components} & \multrow{Satisfaction \\ of ACC} \\
        \hline\hline
        $\Set$ & sets & $2$ & true \\
        $\Pos$ & posets & $2$ & true \\
        $\Mon$ & monoids & $1$ & true \\
        $\Grp$ & groups & $1$ & true \\
        $\Ab$ & abelian groups & $1$ & true \\
        $\mathbf{SGrp}$ & semigroups & infinity & false \\
        $\Ring$ & rings & infinity & false \\
        $\SLat$ & (join-)semilattices & $2$ & true \\
        $\BSLat$ & bounded (join-)semilattices & $1$ & true \\
        $\Lat$ & lattices & $2$ & true \\
        $\BLat$ & bounded lattices & infinity & false \\
        $\Set^n$ & $n$-sorted sets & $2^n$ & true \\
        $\Set^S$ & \multrow{$S$-sorted sets \\ ($S$: an infinite set)} & infinity & false \\
        $\Set_*$ & pointed sets & $1$ & true \\
        $n/\Set$ & sets with $n$ constants & \multrow{the $n$-th \\ Bell number} & true \\
        $S/\Set$ & \multrow{sets with $S$-indexed constants \\ ($S$: an infinite set)} & infinity & false \\
        $\End$ & sets with an endomorphism & infinity & false \\
        $\mathbf{Idem}$ & \multrow{sets with an idempotent \\ endomorphism} & $2$ & true \\
        $\mathbf{Aut}$ & sets with an automorphisms & infinity & false \\
        $\Quiv$ & quivers (or directed graphs) & infinity & false \\
        $\mathbf{RQuiv}$ & reflexive quivers & $2$ & true \\
        $\Cat$ & small categories & $2$ & true \\
        $\Set^\to$ & maps & $3$ & true \\
        $\Cospan$ & cospans of sets & $6$ & true \\
        $\Set^{\bbomega^\op}$ & $\bbomega^\op$-chains of sets & infinity & false \\
        $\Set^{\bbomega}$ & $\bbomega$-chains of sets & infinity & \textbf{true} \\
        $\Set^{\bowtie}$ & see \cref{eg:srg_conn}\ref{eg:srg_conn-doblcospan} & infinity & false \\
        $\Set^\bX$ & see \cref{eg:srg_conn}\ref{eg:srg_conn-X} & infinity & false \\
        $\Set^{\Delta^\op}$ & simplicial sets & $2$ & true \\
        $\Set^{\bG^\op}$ & globular sets & infinity & false \\
        $\mathbf{Nom}$ & nominal sets & infinity & false \\
        $\mathbf{URel}$ & sets with a unary relation & $3$ & true \\
        $\mathbf{BRel}$ & sets with a binary relation & infinity & false \\
        $\mathbf{SRel}$ & sets with a symmetric relation & infinity & false \\
        $\mathbf{RSRel}$ & \multrow{sets with a reflexive \\ symmetric relation} & $2$ & true \\
        $\mathbf{PER}$ & \multrow{sets with a symmetric \\ transitive relation} & $3$ & true \\
        $\mathbf{PreOrd}$ & preordered sets & $2$ & true \\
        $\mathbf{ERel}$ & sets with an equivalence relation & $2$ & true \\
        \hline
    \end{tabular}
    \caption{The ascending chain condition for LFP categories}
    \addcontentsline{toc}{section}{Table A. The ascending chain condition for LFP categories}
    \label{tab:ACC_for_LFP}
\end{table}

\end{document}